\documentclass[11pt,a4paper]{amsart}
\usepackage{amssymb, a4wide, mathdots, url,hyperref,graphicx,ulem, bm}
\usepackage[papersize={210mm, 297mm},textwidth=14.9cm,textheight=21cm,centering]{geometry}
\usepackage[usenames]{color}
\usepackage{enumerate}
\usepackage{wasysym}
\usepackage{dsfont}
\usepackage{tikz-cd}
\usepackage{comment}
\usepackage{MnSymbol}
\usepackage{eqnarray, amsmath}
\usepackage{chngcntr}
\usepackage{apptools}
\AtAppendix{\counterwithin{theorem}{section}}

\newtheorem{theorem}{Theorem}[section]
\newtheorem*{theorem*}{Theorem}
\newtheorem{lemma}[theorem]{Lemma}

\newtheorem{proposition}[theorem]{Proposition}

\theoremstyle{remark}
\newtheorem{remark}{Remark}
\newtheorem{hypothesis}{Hypothesis}

\newcommand{\lcm}{\mathrm{lcm}}

\newcommand{\bP}{\mathbb{P}}

\title{Divisor problems for restricted Fourier coefficients of modular forms}

\author{Yuk-Kam Lau}
\address{
Weihai Institute for Interdisciplinary Research, Shandong University, China \mbox{\rm and} 
Department of Mathematics, The University of Hong Kong, Pokfulam Road, Hong Kong}
\email{yklau@maths.hku.hk}

\author{Wonwoong Lee}
\address{Department of
 Mathematics, The University of Hong
 Kong, Pokfulam Road, Hong Kong}
\email{leeww041@hku.hk}

\date{\today}
\subjclass[2020]{11F30, 11N37}
\keywords{Divisor function, Fourier coefficients of modular forms, Sato-Tate conjecture, Chebotarev density theorem}

\thanks{}

\date{\today}

\begin{document}

\begin{abstract}
    Let $d(n)$ be the number of divisors of $n$. We investigate the average value of $d(a_f(p))^r$ for $r$ a positive integer and $a_f(p)$ the $p$-th Fourier coefficient of a cuspidal eigenform $f$ having integral Fourier coefficients, where $p$ is a prime subject to a constraint on the angle associated with the normalized Fourier coefficient. 
\end{abstract}

\maketitle

\section{Introduction and results}

\subsection{Introduction}

Throughout the paper, we denote by $p$ a rational prime number. The {\it Titchmarsh divisor problem}, which was first studied by Titchmarch in 1930 \cite{Ti30} concerns the following problem: for a sequence of integers $(n(p))_{p \in \bP}$, estimate  
\begin{align*}
    \sum_{p \leq x}d(n(p))
\end{align*}
as precisely as possible, where $d(\cdot)$ denotes the divisor function which counts the number of divisors and $\bP$ is the set of all rational primes. Titchmarch \cite{Ti30} considered this problem for $n(p)=p-a$ for a non-zero integer $a$ and obtained a result under the Grand Riemann Hypothesis ({\it GRH} in short). This result was proved unconditionally by Linnik \cite{Li63}. Their result is 
\begin{align*}
     \sum_{p \leq x}d(p-a) \sim C_ax, \quad \text{ as } x \rightarrow \infty,
\end{align*}
where $C_a$ is a constant depending on $a$. In comparision with the prime number theorem, 
\begin{align*}
    \pi(x):=\sum_{p \leq x}1 \sim \frac{x}{\log x}, \quad \text{ as } x \rightarrow \infty,
\end{align*}
we can see that the size of $d(p-a)$ is $C_a \log p$ on average. Let us express the mean value of a function $f$ on a finite set $M$ as the expected value (of $f(X)$) associated with the random variable $X$ of the uniform distribution on $M$,
\begin{align*}
    \mathbb E_M[f(X)]=\frac{1}{\# M}\sum_{m \in M}f(m). 
\end{align*}
Denoting $\prod(x)=[1,x]\cap \bP$, the result of Titchmarsh and Linnik reads as
\begin{align}\label{eqn_Titchmarsh original}
    \mathbb E_{\prod(x)}[d(X-a)]  \sim C_a \log x \quad \text{ as } x \rightarrow \infty.
\end{align}

In a similar vein, we let ${\rm N}(x)=[1,x]\cap \mathbb N$. One of the most classical results for the average value of $d(n)$ given by Dirichlet in 1838 \cite{Di38} can be stated as
\begin{align}\label{eqn_d(n) average}
    \mathbb E_{{\rm N}(x)}[d(X)]=\log x+(2\gamma-1)+O(x^{-1/2}).
\end{align}
More generally, for $r\in \mathbb N$, we have 
\begin{align}\label{eqn_d^r(n) average}
    \mathbb E_{{\rm N}(x)}[d(X)^r] \sim A_r(\log x)^{2^r-1}, \quad \text{ as } x \rightarrow \infty,
\end{align}
for some positive constant $A_r$, as shown by Wilson \cite{Wi23}. The generalization of this result to the average value of $d(F(x))$ for a polynomial $F$ has been attracting researchers' attention since the 1950s. For example, given an irreducible polynomial $F(x) \in \mathbb Z[x]$, Erd\"{o}s \cite{Er52} proved that
\begin{align}\label{eqn_d(F(n)) average Erdos}
    \mathbb E_{{\rm N}(x)}[d(F(X))] \asymp \log x,
\end{align}
and this was strengthened for quadratic polynomials $F$ to
\begin{align}\label{eqn_d(F(n)) quadratic polynomial}
    \mathbb E_{{\rm N}(x)}[d(F(X))] \sim \lambda \log x,
\end{align}
where $\lambda$ is given as an expression of the Hurwitz class number in \cite{Mc95} and \cite{Mc99}.

Motivated by these problems, there have been many studies on the behavior of the average value of $d(n(p))$ or $d(F(n))$. Akbary and Ghioca \cite{AG12} investigated a geometric version of this problem in the context of abelian varieties. Another result that is closely related to our interests is due to Pollack \cite{Po16}. For an elliptic curve $E$ over $\mathbb Q$, let $\prod_{\rm red}(x)$ be a subset of $\prod (x)$ consisting of $p$ such that $E$ has good reduction at $p$. Pollack conjectured that for a non-CM elliptic curve $E/\mathbb Q$, 
\begin{align}\label{eqn_Pollack conjecture d(E(f_p))}
    \mathbb E_{{\prod}_{\rm red}(x)} [d\left(\#E(\mathbb F_X)\right)] \sim c_E \log x, \quad \text{as } x \rightarrow \infty
\end{align}
for some constant $c_E>0$, and he proved that
\begin{align}\label{eqn_Pollack's theorem for ell}
    \mathbb E_{{\prod}_{\rm red}(x)}[d\left(\#E(\mathbb F_X)\right)] \asymp_E \log x
\end{align}
for $x \gg_E 1$ under GRH. Here, $\mathbb F_p$ denotes the finite field  of order $p$ up to isomorphism. Pollack used a refinement of Erd\"{o}s's idea in \cite{Er52}.

Pollack's argument has been generalized to cuspidal eigenforms by Chiriac \cite{Ch22}. Chiriac proved under GRH that if $f$ is a non-CM newform of weight $k \geq 2$ with integral Fourier coefficients $a_f(n)$, then for $\prod(x;f)$ a subset of $\prod(x)$ consisting of $p$ such that $a_f(p) \neq 0$, we have
\begin{align}\label{eqn_Chi upper bound}
    \mathbb E_{\prod(x;f)}[d(a_f(X))] \ll \log x.
\end{align}
This result improved upon the conditional upper estimate below by Gun and Murty \cite{GM14}:  under GRH, 
\begin{align*}
    \mathbb E_{\prod(x;f)}[d(a_f(X))] \ll (\log x)^{A_f},
\end{align*}
where $A_f$ is a constant depending on $f$. It is worthy to note that Chiriac's theorem is valid for a more general setting (\cite[Theorem 1.1]{Ch22}).

 Gun and Murty \cite{GM14} also obtained the lower estimate
\begin{align}\label{eqn_GM lower bound}
    \log x \ll \mathbb E_{\prod(x;f)}[d(a_f(X))],
\end{align}
assuming GRH. Combining \eqref{eqn_Chi upper bound} and \eqref{eqn_GM lower bound}, we obtain the analogue of \eqref{eqn_Pollack's theorem for ell} for $f$. Likewise for elliptic curves, Gun and Murty \cite{GM14} asked whether it is reasonable to expect 
\begin{align}\label{eqn_Gun Murty conjecture d(a_f(p))}
    \mathbb E_{\prod(x;f)}[d(a_f(X))] \sim B_f\log x, \quad \text{ as } x \rightarrow \infty
\end{align}
for some constant $B_f$ depending on $f$.

\subsection{Main result}\label{SMR} In this paper, we consider similar divisor problems concerning cuspidal eigenforms with a restriction on the angle associated with the normalized Fourier coefficients. 

Let $f$ be a cuspidal eigenform of weight $k\geq 2$ and level $N$, and let $a_f(n)$ be its $n$-th Fourier coefficient. Suppose that $a_f(n)$ are integers for all $n \in \mathbb N$. By the Ramanujan-Petersson bound for holomorphic newforms, which was proved by Deligne, we can express
\begin{align}\label{fc}
    \lambda_f(p):=a_f(p)/p^{\frac{k-1}{2}} = 2\cos \theta_f(p) \in [-2,2]
\end{align}
for some $\theta_f(p) \in [0,\pi]$. Let $I \subseteq [0,\pi]$ be an interval. By the Sato-Tate conjecture, which was proved in \cite{BGHT11} for non-CM holomorphic newforms, we have
\begin{align}\label{eqn_Sato Tate law}
    \#\{p \leq x : \theta_f(p) \in I\} \sim \mu_{ST}(I)\pi(x), \quad \text{ as } x \rightarrow \infty,
\end{align}
where
\begin{align*}
    \mu_{ST}(I):=\frac{2}{\pi}\int_I \sin^2 \theta d\theta 
\end{align*}
is the {\it Sato-Tate measure}. Denote by $\prod (x,I;f)$ the set of primes $p \leq x$ such that $a_f(p) \neq 0$ and $\theta_f(p) \in I$.
We shall prove the following. 
\begin{theorem}\label{thm_main theorem for modular forms}
    Let $r \in \mathbb N$ and $f$ be a non-CM newform of weight $k \geq 2$ with integer Fourier coefficients $a_f(n)$, $n\in \mathbb N$. Under {\it Hypothesis}~\ref{hypo_nice analytic properties of L-function} in \S~\ref{sec_L ftns on sym otimes chi} and GRH, 
        \begin{align}\label{t1}
            \mathbb E_{\prod (x,I;f)}[d(a_f(X))^r] \asymp_{f,r} (\log x)^{2^r-1}, 
        \end{align}
        as $x \rightarrow \infty$. Moreover, if $I=[0,\pi]$, then the same conclusion holds under GRH only.
\end{theorem}

Here, {\it Hypothesis}~\ref{hypo_nice analytic properties of L-function} concerns certain analytic properties of the $L$-functions $L(s,\mathrm{Sym}^m f \otimes \chi)$, which are expected  under the Langlands reciprocity conjecture; see \S\ref{sec_L ftns on sym otimes chi} for details.

\begin{remark}
    \label{conj_main conjecture} 
    (1) The constants in $\asymp_{f,r}$ are independent of $I$.
    (2) As in \eqref{eqn_Gun Murty conjecture d(a_f(p))}, 
    we speculate the following asymptotic formula: as $x \rightarrow \infty$, 
    \begin{align*}
        \mathbb E_{\prod (x,I;f)}[d(a_f(X))^r] \sim B_{f,r,\mu_{ST}(I)}(\log x)^{2^r-1} 
    \end{align*}
    for some  constant $B_{f,r,\mu_{ST}(I)}>0$ depending on $f$, $r$, and $\mu_{ST}(I)$.
This is currently out of our reach.
\end{remark}

\begin{remark}
For the proof of Theorem~\ref{thm_main theorem for modular forms}, the specific instances of GRH that we require are as follows. We need GRH for Artin $L$-functions, which is used to derive the effective Chebotarev theorem and to prove \eqref{eqn_a(p)=0 proportion}. Moreover, for every natural number $\delta$, every finite Galois group  
$G = \mathrm{Gal}(K/\overline{\mathbb Q}) \subseteq \mathrm{GL}_2(\mathbb Z/\delta)$, 
and every irreducible representation $\rho$ of $G$ with character $\chi$ (including the trivial character), we assume GRH for $L(s,\mathrm{Sym}^m f \otimes \chi)$.  
\end{remark}

To establish Theorem~\ref{thm_main theorem for modular forms}, we shall need two auxiliary results, which are of independent interests. 
First, as in \cite{Ch22} and \cite{Po16}, we extend Erd\"os' result for our purposes but in a somewhat more general setting. Second, we invoke an effective version of the Chebotarev-Sato-Tate conjecture,  conditionally deduced from Hypothesis~\ref{hypo_nice analytic properties of L-function} and GRH.   

\subsection{Auxiliary results}\label{SAR} Let $\{A_x\}_{x\in (0,\infty)}$ be a collection of finite sets such that $A_i \subset A_j$ for $i\le j$ and  $A:= \bigcup_x A_x$ is infinite. Let $n:A\to \mathbb N$ be a function, and for an integer $\delta \ge 1$, define
    \begin{align*}
        \pi(x,\delta):=\#\{a \in A_x: n(a) \equiv 0 \text{ (mod } \delta)\}.
    \end{align*}
Consider the following conditions.  
\begin{itemize}
    \item (U1) There exists $\beta>0$ such that $|n(a)|<x^{\beta}$ for every $a \in A_x$ and every $x>0$. 
    \item (U2) There is a positive real number $c_0$ such that 
    \begin{align*}
        \pi(x,\delta) \ll \frac1{\phi(\delta)}\# A_x
    \end{align*}
    for all integers $1\le \delta \leq x^{c_0}$ and for all $x>0$, where $c_0$ and the implied constant are independent of $A$. Here, as usual $\phi$ denotes the Euler totient function. 
    \item (L) There is a positive real number $c_1$ such that 
    \begin{align*}
        \pi(x,\delta) \gg \frac1{\delta}\# A_x
    \end{align*}
    for all integers $1\le \delta \leq x^{c_1}$ and for all sufficiently large $x$, where $c_1$ and the implied constant are independent of $A$.
\end{itemize}

\begin{theorem}\label{thm_main theorem general}
    Let the set $A$ and the function $n$ be defined as above and $r\ge 1$ be any integer. 
    \begin{enumerate}
        \item Assume {\normalfont (U1)} and  {\normalfont (U2)}. We have
        \begin{align*}
             \sum_{a \in A_x }d(n(a))^r \ll \#A_x\cdot (\log x)^{2^r-1},
        \end{align*}
        where the implied constant is independent of $A$ and $x$. 
        \item Assume  {\normalfont (L)}. For all sufficiently large $x$, we have
        \begin{align*}
             \sum_{a \in A_x }d(n(a))^r \gg \#A_x\cdot (\log x)^{2^r-1},
        \end{align*}
        where the implied constant is independent of $A$ and $x$. 
    \end{enumerate}
\end{theorem}

The strategy of the proof is based on combinatorial arguments from \cite{Er52}, together with 
tools built with the multivariable Tauberian theorem due to Zavala \cite{Za22}  and Wolke's work \cite{Wo71}. 

As well, we need an effective version of the Chebotarev-Sato-Tate conjecture. This was proved conditionally by Wong \cite{Wo19} in a quite general setting. Theorem~\ref{prop_P(x,I;f;delta)} below is a specialization, though being a quite easy consequence of Wong's work, whose proof is explained in \S~\ref{sec_effective CST} for completeness. 
When $I=[0,\pi]$, we will use the effective Chebotarev density theorem proved by Lagarias and Odlyzko \cite{LO77}, for which only GRH is required.

\begin{theorem}\label{prop_P(x,I;f;delta)}
Let $f$ be a non-CM newform of integral weight of level $N$. Suppose the Fourier coefficients of $f$ are integers. For an interval $I \subseteq [0,\pi]$ and any positive integer $\delta$, under GRH (when $I=[0,\pi]$), or GRH and Hypothesis A (when $I\subset [0,\pi]$), we have
    \begin{align*}
    & \#\{p\le x: 0\neq a_f(p)\equiv 0 \ ({\rm mod} \ \delta), \ \theta_f(p)\in I\}  \\
      &  =h_f(\delta)\mu_{ST}(I)\pi(x)+O\left(x^{3/4}h_f(\delta)\delta^6\log^{1/2}(\delta N)\right)+O\left(x^{3/4}h_f(\delta)\delta^4 \log^{1/2}x\right),
    \end{align*}
    where $\lambda_f(p)$ and $\theta_f(p)$ are defined as in \eqref{fc} and the constant $h_f(\delta)$ is a certain constant depending on $f$ and $\delta$ that satisfies
    $$
    h_f(\delta) \asymp_f \prod_{\substack{\ell^m\| \delta\\  \ell \gg_f 1}} \ell^{2-m}(\ell^2-1)^{-1}.
    $$
\end{theorem}
\begin{remark} Hypothesis A is predicted to be valid in the context of Langlands program. Precisely, it will follow from the automophy of $L(s, {\rm Sym}^mf \times \chi)$, see \S~\ref{sec_L ftns on sym otimes chi}. 
\end{remark}
The outline of the paper is as follows. In $\mathsection$~\ref{sec_multivariable Tauberian}, we introduce the multivariable Tauberian theorem due to Zavala \cite{Za22} and provide its direct applications (Proposition~\ref{prop_1/phi(lcm) formula}, \ref{prop_lcm/phi(lcm)^2 formula}). Some results in multiplicative number theory are necessary, which are collected in \S~\ref{pre}. Then we apply the results in these two sections to prove Theorem~\ref{thm_main theorem general} in $\mathsection$~\ref{sec_proof of main thm in general}. Next, we provide a review on $L$-functions and discuss {\it Hypothesis}~\ref{hypo_nice analytic properties of L-function} in $\mathsection$~\ref{sec_preliminaries on L-functions}. The work on effective Chebotarev-Sato-Tate conjectures, leading to Theorem~\ref{prop_P(x,I;f;delta)}, is explored in  $\mathsection$~\ref{sec_effective CST}. Finally we prove Theorem~\ref{thm_main theorem for modular forms} in $\mathsection$~\ref{sec_proof of main thm for modular forms}. Appendix~\ref{sec_appendix} describes the adelic representation attached to modular forms.

\section{Multivariable Tauberian theorem}\label{sec_multivariable Tauberian}

This section is devoted to introducing the multivariable Tauberian theorem and its applications, which will be used to prove Theorem~\ref{thm_main theorem general}. We begin by introducing some notations. Let $\lcm(m_1,\ldots,m_r)$ denote the least common multiple of $m_1,\ldots,m_r$. For an $r$-tuple $x=(x_1,\ldots,x_r) \in \mathbb R^r$, we define $|x|_1:=|x_1|+\cdots+|x_r|$. We also define $\langle x,y \rangle:=\sum_{i=1}^r x_iy_i$ for $x=(x_1,\ldots,x_r) \in \mathbb R^r$ and $y=(y_1,\ldots,y_r) \in \mathbb R^r$. A multivariable function $f \colon \mathbb N^r \to \mathbb C$ is said to be multiplicative if for all $m=(m_1,\ldots,m_r) \in \mathbb N^r$ and $m'=(m_1',\ldots,m_r') \in \mathbb N^r$ satisfying $$\mathrm{gcd}(\lcm(m_1,\ldots,m_r), \lcm(m_1',\ldots,m_r'))=1,$$ 
we have $f(m_1m_1',\ldots,m_r m_r')=f(m)f(m')$. 

Let $(g,\kappa, c, \eta)$ denote a quadraple consisting of the following data:
\begin{itemize}
    \item A function $g \colon \mathbb N_0^r \to \mathbb C$ that has subexponential growth, i.e., for any $\epsilon>0$, we have $g(\nu) \ll_{\epsilon} e^{\epsilon|\nu|_1}$ uniformly in $\nu \in \mathbb N_0^r$.
    \item A function $\kappa \colon \mathbb N_0^r \to [1,\infty) \cup \{0\}$ that satisfies $\kappa(0)=0$ and $\mathrm{inf}_{\nu \in \mathbb N_0^r \setminus \{0\}} \frac{\kappa(\nu)}{|\nu|_1}>0$.
    \item An $r$-tuple $c=(c_1,\ldots,c_r) \in [0,\infty)^r$.
    \item A real number $\eta \in  (0,\infty)$.
\end{itemize}
We define
\begin{align*}
    \mathcal I&:=\{\nu \in \mathbb N_0^r: \kappa(\nu)=1 \text{ and } g(\nu) \neq 0\}, \\
    \mathcal J&:=\{{\bm e_i}: c_i=0\}, 
\end{align*}
where $({\bm e_1},\ldots,{\bm e_r})$ is the canonical basis for $\mathbb R^r$. For such $\mathcal I$ and $\mathcal J$, we let
\begin{align*}
    \rho:=\sum_{\nu \in \mathcal I}g(\nu)+\# \mathcal J-\mathrm{rk}(\mathcal I \cup \mathcal J),
\end{align*}
where $\mathrm{rk}(\mathcal I \cup \mathcal J)$ is the rank of the subspace of $\mathbb R^r$ generated by $\mathcal I \cup \mathcal J$. 

\begin{theorem}{\cite[Theorem 16]{Za22}}\label{thm_multivariable Tauberian}
    Let $f \colon \mathbb N^n \to \mathbb C$ be a multiplicative function that satisfies the condition as follows; for any $\epsilon>0$, 
    \begin{align*}
        f(p^{\nu_1},\ldots,p^{\nu_r})-g(\nu)p^{\langle c,\nu \rangle-\kappa(\nu)} \ll_{\epsilon} e^{\epsilon |\nu|_1}p^{\langle c,\nu \rangle-\kappa(\nu)-\eta},
    \end{align*}
    uniformly in $\nu \in \mathbb N_0^n$ and $p$ a prime number. Assume $\mathcal I$ is non-empty. Then there exists a polynomial $Q_{\infty}$ of degree at most $\rho$ and a positive constant $\mu_{\infty}>0$ such that
    \begin{align*}
        \sum_{d_1,\ldots,d_r \leq x}f(d_1,\ldots,d_r)=x^{|c|_1}Q_{\infty}(\log x)+O(x^{|c|_1-\mu_{\infty}}) \quad \text{as } x \rightarrow \infty. 
    \end{align*}
    Furthermore, if we assume that
    \begin{enumerate}[\normalfont(i)]
        \item $\mathrm{rk}(\mathcal I \cup \mathcal J)=r$,
        \item the point $(1,\ldots,1) \in \mathbb R^r$ belongs to the interior of the cone generated by $\mathcal I \cup \mathcal J$, i.e.,
        \begin{align*}
            (1,\ldots,1) \in \bigg\{\sum_{\nu \in \mathcal I \cup \mathcal J}\lambda_{\nu}\nu : \lambda_{\nu} \in (0,\infty) \text{ for any } \nu \in \mathcal I \cup \mathcal J\bigg\},
        \end{align*}
    \end{enumerate}
    then there exists $C>0$ such that
    \begin{align*}
        \sum_{d_1,\ldots,d_r \leq x}f(d_1,\ldots,d_r)=Cx^{|c|_1}(\log x)^{\rho}+O((\log x)^{\rho-1}), \quad \text{as } x \rightarrow \infty.
    \end{align*}
\end{theorem}

\begin{proposition}\label{prop_1/phi(lcm) formula}
    There exists a constant $C_r>0$ such that
    \begin{align*}
        \sum_{d_1, \ldots, d_r \leq x}\frac{1}{\phi(\lcm(d_1,\ldots,d_r))}=C_r(\log x)^{2^r-1}+O((\log x)^{2^r-2}).
    \end{align*}
\end{proposition}

\begin{proof}
    We set $g=1, \kappa(\nu)=\max  \nu_i, c=(0,\ldots,0), \eta=1$, and
    \begin{align*}
        f(d_1,\ldots,d_r)=\frac{1}{\phi(\lcm(d_1,\ldots,d_r))}
    \end{align*} 
    so that $f$ is a multiplicative function. We have
    \begin{align*}
        \mathcal I&=\{(\delta_1,\ldots,\delta_r): \delta_i=0 \text{ or } 1\} \setminus \{(0,\ldots,0)\}, \\
        \mathcal J&=\{{\bm e_1}, \ldots, {\bm e_r}\},
    \end{align*}
    so that
    \begin{align*}
        \mathrm{rk}(\mathcal I \cup \mathcal J)=r, \quad |\mathcal J|=r, \quad \sum_{\nu \in \mathcal I}g(\nu)=2^r-1,
    \end{align*}
    and consequently, $\rho=2^r-1$. Since $\mathcal I \cup \mathcal J$ generates the whole space $\mathbb R^r$, the condition (ii) in Theorem~\ref{thm_multivariable Tauberian} obviously holds. Thus, it is enough to show that
    \begin{align}\label{eqn_quadraple data class condition}
        f(p^{\nu_1},\ldots,p^{\nu_r})-g(\nu)p^{\langle c,\nu \rangle-\kappa(\nu)} \ll_{\epsilon} {\rm e}^{\epsilon |\nu|_1}p^{\langle c,\nu \rangle-\kappa(\nu)-\eta}
    \end{align}
    uniformly in $\nu \in \mathbb N_0^n$ and $p$ a prime number. Indeed, we have
    \begin{align*}
        f(p^{\nu_1},\ldots,p^{\nu_r})=p^{-\max \nu_i}\left(1-\frac{1}{p}\right)^{-1}.
    \end{align*}
    and hence the left-hand side of \eqref{eqn_quadraple data class condition} satisfies
    \begin{align*}
        f(p^{\nu_1},\ldots,p^{\nu_r})-g(\nu)p^{\langle c,\nu \rangle-\kappa(\nu)}=p^{-\max \nu_i}\left(\frac{1}{p}+\frac{1}{p^2}+\cdots\right) \leq 2p^{-\max \nu_i-1} \leq 2e^{\epsilon |\nu|_1}p^{\langle c,\nu \rangle-\kappa(\nu)-\eta}.
    \end{align*}
    We complete the proof by applying Theorem~\ref{thm_multivariable Tauberian}.
\end{proof}

\begin{proposition}\label{prop_lcm/phi(lcm)^2 formula}
    There exists a constant $D_r>0$ such that
    \begin{align*}
        \sum_{d_1, \ldots, d_r \leq x}\frac{\lcm(d_1,\ldots,d_r)}{\phi(\lcm(d_1,\ldots,d_r))^2}=D_r(\log x)^{2^r-1}+O((\log x)^{2^r-2}).
    \end{align*}
\end{proposition}

\begin{proof}
    We use the same quadraple $(g,\kappa,c,\eta)$ as in the proof of Proposition~\ref{prop_1/phi(lcm) formula}, and let
    \begin{align*}
        f(d_1,\ldots,d_r)=\frac{\lcm(d_1,\ldots,d_r)}{\phi(\lcm(d_1,\ldots,d_r))^2}.
    \end{align*}
    All the conditions required to apply Theorem~\ref{thm_multivariable Tauberian} are given in the proof of Proposition~\ref{prop_1/phi(lcm) formula} except for \eqref{eqn_quadraple data class condition}. For this remaining condition, note that
    \begin{align*}
        \frac{\lcm(p^{\nu_1},\ldots,p^{\nu_r})}{\phi(\lcm(p^{\nu_1},\ldots,p^{\nu_r}))^2}=p^{-\max \nu_i}\left(1-\frac{1}{p}\right)^{-2}.
    \end{align*}
    Thus, we have
    \begin{align*}
        f(p^{\nu_1},\ldots,p^{\nu_1})-g(\nu)p^{\langle c,\nu \rangle-\kappa(\nu)}=p^{-\max \nu_i}\left(1-\left(1-\frac{1}{p}\right)^{-2}\right) \ll p^{-\max \nu_i-1} \leq {\rm e}^{\epsilon|\nu|_1}p^{-\max \nu_i-1},
    \end{align*}
    which completes the proof.
\end{proof}

The same quadraple $(g,\kappa,c,\eta)$ as in Proposition~\ref{prop_1/phi(lcm) formula} or \ref{prop_lcm/phi(lcm)^2 formula} gives a similar estimation for the sum of $\frac{1}{\lcm(d_1,\ldots,d_r)}$ (see \cite[Corollary 5]{Za22}). Specifically, we have
\begin{align}\label{eqn_1/lcm formula}
    \sum_{d_1, \ldots, d_r \leq x}\frac{1}{\lcm(d_1,\ldots,d_r)}=B_r(\log x)^{2^r-1}+O((\log x)^{2^r-2})
\end{align}
for some $B_r>0$.

\section{Preliminaries for Theorem~\ref{thm_main theorem general}}\label{pre}

For a positive integer $m$, we denote the largest prime factor of $m$ by $q_m$.  

\begin{lemma}\label{lem_Wolke lemma}
    Let $c \in (0,1)$ and $F \colon \mathbb N \to \mathbb R$ be a non-negative multiplicative function such that there exists a positive constant $\alpha_1$ satisfying
    \begin{align*}
        F(p^e) &\leq \alpha_1^e, \\
        F(m) &\ll_{\epsilon} m^{\epsilon}
    \end{align*}
     for any prime $p$, any positive integers $e$ and $m$, and any positive real number $\epsilon$. If $1\leq s \leq \frac{\log x}{\log\log x}$, then
    \begin{align*}
        \sum_{\substack{m \geq x^{c/4} \\ q_m \leq x^{1/s}}} \frac{F(m)}{m} \ll \exp \bigg(\sum_{p \leq x^{1/s}} \frac{F(p)}{p}-\frac{cs\log s}{16}\bigg),
    \end{align*}
    where the implied constant depends at most on $c$ and $\alpha_1$. 
\end{lemma}

\begin{proof}
    The assertion has been essentially proved in \cite[Lemma~3]{Wo71} or \cite[Lemma~4]{Sh80}. In the proof of \cite[Lemma~4]{Sh80}, it is proved that
    \begin{align*}
        \sum_{\substack{m \geq X \\ q_m \leq Y}}\frac{F(m)}{m} \ll \exp\bigg((\delta-1)\log X+\sum_{p \leq Y}\frac{F(p)}{p}+2\alpha_1 Y^{1-\delta}\bigg)
    \end{align*}
    for any $3/4 \leq \delta \leq 1$. Taking $X=x^{c/4},Y=x^{1/s}$, and $\delta=1-\frac{s\log s}{4\log x}$, which satisfies $3/4\leq\delta\leq 1$ since $1\leq s\leq \frac{\log x}{\log\log x}$, we obtain
    \begin{align*}
        \sum_{\substack{m \geq x^{c/4} \\ q_m \leq x^{1/s}}} \frac{F(m)}{m} &\ll \exp \left(-\frac{cs\log s}{8}+\sum_{p \leq x^{1/s}} \frac{F(p)}{p}+2\alpha_1 s^{1/4}\right) \\
        &\ll \exp \left(\sum_{p \leq x^{1/s}} \frac{F(p)}{p}-\frac{cs\log s}{16}\right).
    \end{align*}
    The last inequality of the above follows from $cs\log s + c^{-1}\gg 2\alpha_1 s^{1/4}$.
\end{proof}

Let $S(X,Y)$ be the set of all positive integers $n\le X$ such that $n$ is {\it $Y$-smooth}, i.e., free of any prime divisors larger than $Y$. 
\begin{lemma}\label{lem_Chiriac's lemma}
    Let $r\in \mathbb N$ and $c \in (0,1)$. If $x^{1/s} \leq (\log x)^2$, then
    \begin{align*}
        \sum_{\substack{m \in S(x^{rc},x^{1/s}) \\ m \geq x^{c/4}}}\frac{1}{m} \ll x^{-\delta}
    \end{align*}
    for some $\delta>0$, where the implied constant and $\delta$ are independent of $s$ and $x$.
\end{lemma}

\begin{proof}
    This is a part of \cite[Lemma~2.1]{Ch22}. The condition for the range of $m$ differs slightly from those in \cite[Lemma~2.1]{Ch22}, but the proof remains applicable. 
\end{proof}

Given any positive integer $r$. Define for any $m\in \mathbb N$, and any $1\le Z\le X$, 
\begin{align*}
    \Phi_r(m;X,Z):= \#\{(d_1,\ldots,d_r) \in \mathbb N^r: \ \lcm(d_1,\ldots,d_r)=m, \  Z \leq d_1, \ldots, d_r \leq X\}.
\end{align*}
\begin{proposition}\label{prop_lcm number hypo for alpha is 2 to the r}
    Let $r$ be a positive integer, $c \in (0,1)$ and $s \geq 4/c$ be a real number.
    
    If $x^{1/s}>(\log x)^2$, then
        \begin{align*}
            \sum_{\substack{m \in S(x^{rc},x^{1/s}) \\ m \geq x^{c/4}}}\frac{\Phi_r(m;x^c,x^{c/4})}{m} \ll (\log x)^{2^r-1}\exp \left(-\frac{cs\log s}{16}\right).
        \end{align*}
        
    If $x^{1/s}\leq (\log x)^2$, then the above sum is $\ll x^{-\delta}$ for some $\delta>0$. In both cases, the implied constants in $\ll$ depend at most on $r$ and $c$.
\end{proposition}

\begin{proof}
    First, suppose $x^{1/s} > (\log x)^2$, i.e., $s < \frac{\log x}{2\log\log x}$. Note that $\Phi_r(m;x^c,x^{c/4})\le \Phi_r(m):=\Phi_r(m;\infty,0)$, the number of $r$-tuples $(d_1,\ldots,d_r)$ with $\lcm(d_1,\ldots,d_r)=m$. Clearly $\Phi_r(m)$ is a multiplicative function, and  for any positive integer $e$, 
    \begin{align*}
        \Phi_r(p^e) 
        =&\#\{ (d_1,\ldots,d_r)\in \mathbb{N}^r: \ \lcm(d_1,\ldots,d_r)=p^i  \text{ for some } i \in \{0,\ldots,e\}\} \\
        &\mbox{ } -\#\{ (d_1,\ldots,d_r)\in \mathbb{N}^r: \ \lcm(d_1,\ldots,d_r)=p^i  \text{ for some } i \in \{0,\ldots,e-1\}\} \\
        =& (e+1)^r-e^r \ll \alpha_1^e
    \end{align*}
    for some constant $\alpha_1$ depending on $r$. Therefore, by Lemma~\ref{lem_Wolke lemma}, we obtain
    \begin{align*}
        \sum_{\substack{m \in S(x^{rc},x^{1/s}) \\ m \geq x^{c/4}}}\frac{\Phi_r(m;x^c,x^{c/4})}{m} \leq \sum_{\substack{m \geq x^{c/4} \\ q_m \leq x^{1/s}}} \frac{\Phi_r(m)}{m}
        \ll \exp\left(\sum_{p \leq x^{1/s}}\frac{2^r-1}{p}-\frac{cs\log s}{16}\right).
    \end{align*}
    Since $\sum_{p \leq Y} p^{-1}  = \log\log Y+ O(1)$, we get the desired result for the case $x^{1/s} > (\log x)^2$.

    Now suppose $x^{1/s} \leq (\log x)^2$. We use the Cauchy-Schwarz inequality to get
    \begin{align*}
        \left(\sum_{\substack{m \in S(x^{rc},x^{1/s}) \\ m \geq x^{c/4}}}\frac{\Phi_r(m;x^c,x^{c/4})}{m}\right)^2 &\leq \sum_{\substack{m \in S(x^{rc},x^{1/s}) \\ m \geq x^{c/4}}}\frac{\Phi_r(m)^2}{m} \sum_{\substack{m \in S(x^{rc},x^{1/s}) \\ m \geq x^{c/4}}} \frac{1}{m}.
    \end{align*}
    Since $m$ is an $x^{1/s}$-smooth number, we evaluate the first sum on the right-hand side as follows:
    \begin{align*}
    \sum_{\substack{m \in S(x^{rc},x^{1/s}) \\ m \geq x^{c/4}}}\frac{\Phi_r(m)^2}{m}  
        &\leq  \prod_{p \leq x^{1/s}}\left(1+\frac{\Phi_r(p)^2}{p}+\frac{\Phi_r(p^2)^2}{p^2}+\cdots\right)  \\ 
        &=\prod_{p \leq x^{1/s}}\left(1+\frac{(2^r-1)^2}{p}+O\left(\frac{1}{p^2}\right)\right)\\
        &=\exp\left(\sum_{p \leq x^{1/s}} \log\left(1+\frac{(2^r-1)^2}{p}+O\left(\frac{1}{p^2}\right)\right)\right) \\ 
        &=\exp\left(\sum_{p \leq x^{1/s}}\frac{(2^r-1)^2}{p}+O(1) \right)\\ 
        &\ll (\log x)^{(2^r-1)^2}.
    \end{align*}
     Applying Lemma~\ref{lem_Chiriac's lemma} to the second sum yields the desired result.
\end{proof}

\begin{proposition}\label{prop4.4}
    Let $r$ be a positive integer, $c \in (0,1)$ and $x\ge 3$. 
    If $4/c \leq s < \frac{\log x}{2\log\log x}$, then we have
    \begin{align*}
        \sum_{\substack{d_1,\ldots,d_r \in S(x^c,x^{1/s}) \\ d_1,\ldots,d_r \geq x^{c/4}}}\frac{1}{\lcm(d_1,\ldots,d_r)} \ll (\log x)^{2^r-1}\exp \left(-\frac{cs\log s}{16}\right).
    \end{align*}
    If $s \geq \frac{\log x}{2\log\log x}$, then the above sum is $\ll x^{-\delta}$ for some $\delta>0$. In both cases, the implied constant depends at most on $r$ and $c$.
\end{proposition}
\begin{proof}
For any $x^{1/s}$-smooth integers $d_1,\ldots, d_r$, $m= \lcm(d_1,\ldots, d_r)$  is $x^{1/s}$-smooth, and the converse is clearly true. Therefore,
    \begin{align*}
        \sum_{\substack{d_1,\ldots,d_r \in S(x^c,x^{1/s}) \\ d_1,\ldots,d_r \geq x^{c/4}}}\frac{1}{\lcm(d_1,\ldots,d_r)}=\sum_{\substack{m \in S(x^{rc},x^{1/s}) \\ m \geq x^{c/4}}}\frac{\Phi_r(m;x^c,x^{c/4})}{m},
    \end{align*}
where the (extra) constraint $ m \geq x^{c/4}$ is implied by  $d_1,\ldots, d_r\ge x^{c/4}$.
The result follows from Proposition~\ref{prop_lcm number hypo for alpha is 2 to the r}.
\end{proof}

\section{Proof of Theorem~\ref{thm_main theorem general}}\label{sec_proof of main thm in general}

   \subsection{Lower bound}    
    Choose $c'= c_1/r$, where $c_1$ is the number in (L). Plainly,  
    \begin{align*}
        \sum_{a\in A_x} d(n(a))^r = \sum_{a\in A_x} \sum_{d_1,\ldots,d_r \mid n(a)} 1 \geq \sum_{a\in A_x} \sum_{\substack{d_1,\ldots,d_r \leq x^{c'} \\ d_1,\ldots,d_r \mid n(a)}} 1=\sum_{\substack{d_1,\ldots,d_r \leq x^{c'}\\ \delta = \lcm (d_1,\ldots,d_r)} } \pi(x,\delta).
    \end{align*}
By (L) and \eqref{eqn_1/lcm formula}, for all sufficiently large $x$, the right-hand side is bounded below by
    \begin{align*}
 \#A_x\sum_{d_1,\ldots,d_r \leq x^{c'}}\frac{1}{\lcm(d_1,\ldots,d_r)} 
        \gg  \#A_x(\log x)^{2^r-1}.
    \end{align*}
    This completes the proof.

\subsection{Upper bound}    
    Let $r\ge 1$ be any given integer and $x>0$ be any number.  
    Let $n(a)=p_1\cdots p_J$ be a prime factorization of $n(a)$ with $p_1 \leq \cdots \leq p_J$. Denote $j \leq J$ to be the largest index such that $p_1 \cdots p_j \leq x^c$, where $c:=c_0/r$ and $c_0$ is the constant in (U2). (Both $J$ and $j$ depend on $a$ and we suppress the dependence for simplicity). Decompose 
    \begin{align*}
        \sum_{a\in A_x} d(n(a))^r= \sum_{\substack{a \in A_x \\ j=0}}d(n(a))^r+\sum_{\substack{a\in A_x\\ j>0}}d(n(a))^r=: \Sigma_1 + \Sigma_2, \mbox{ say}.
    \end{align*}
    
    When $j=0$, i.e., $p_1,\ldots,p_J>x^c$, we have
    \begin{align*}
        x^{Jc}<  |n(a)|\le x^{\beta}
    \end{align*}
    by (U1). Thus $J<\beta/c$ and this implies $d(n(a))^r\le 2^{r^2\beta /c_0}$. The first sum $\Sigma_1$ is $O(\# A_x)$.

    We turn to the second sum $\Sigma_2$, which can be expressed as 
    \begin{align*}
        \sum_{\substack{a \in A_x \\ J-j < (2\beta+1)/c}}d(n(a))^r+\sum_{\substack{a\in A_x \\ J-j \geq (2\beta+1)/c}}d(n(a))^r=:S_1+S_2.
    \end{align*}
    (The condition $j>0$ will be tacitly assumed.)
    If $J-j < (2\beta+1)/c$, then
    \begin{align*}
        d(p_1\cdots p_J)^r \leq d(p_1\cdots p_j)^r2^{(2\beta+1)r/c}  \ll \sum_{d_1,\ldots,d_r \mid p_1\cdots p_j}1 \leq \sum_{\substack{d_1,\ldots,d_r \leq x^c \\ d_1,\ldots,d_r \mid n(a)}}1.
    \end{align*}
    Therefore,
    \begin{align*}
        S_1& \ll \sum_{\substack{a\in A_x \\ J-j < (2\beta+1)/c}}\sum_{\substack{d_1,\ldots,d_r \leq x^c \\ d_1,\ldots,d_r \mid n(a)}}1 \le \sum_{d_1,\ldots,d_r \leq x^c}\sum_{\substack{a\in A_x  \\ n(a) \equiv 0\pmod{\delta}\\ \delta =  \lcm(d_1,\ldots,d_r)} }1 = \sum_{\substack{d_1,\ldots,d_r \leq x^c\\ \delta =  \lcm(d_1,\ldots,d_r)} }\pi(x,\delta).
    \end{align*}
    Note that $\delta \le x^{rc} = x^{c_0}$. By (U2) and Proposition~\ref{prop_1/phi(lcm) formula}, the resulting sum on the right-hand side is 
    \begin{align*}
        \ll \sum_{d_1,\ldots,d_r \leq x^c}\frac{\#A_x}{\phi(\lcm(d_1,\ldots,d_r))} \ll_{c_0,r} \#A_x(\log x)^{2^r-1}.
    \end{align*}
    
    Now, we estimate $S_2$. If $J-j \geq (2\beta+1)/c$, then we have $p_{j+1}<x^{c/2}$ because, otherwise, i.e., $p_{j+1} \geq x^{c/2}$, we have 
    \begin{align*}
        |n(a)| \geq p_{j+1}\cdots p_J \geq x^{c(J-j)/2} \geq x^{\beta+1/2},
    \end{align*}
    which contradicts (U1). Moreover, we have $p_1\cdots p_j > x^{c/2}$ since $p_1\cdots p_{j+1} \ge x^c$ by our choice of $j$ and the bound $p_{j+1}<x^{c/2}$. Recall $p_J\ge \cdots \ge p_j\ge \cdots \ge p_1$. 
    Choose the positive integer $s=s(p_j)$ so that 
    \begin{align*}
        x^{1/(s+1)} \leq p_j <x^{1/s}.
    \end{align*}
    Then $p_{j+1}\cdots p_J \geq x^{\frac{J-j}{s+1}}$, and consequently $J-j \leq (s+1)\beta$ by (U1). We split $S_2$ into two sums 
    \begin{align*}
        S_2:= S_3 + S_4
    \end{align*}
    according as $s < \frac{\log x}{2\log\log x}$ (equivalently, $x^{1/s}>(\log x)^2$) or not. 

    The sum $S_3$ runs over $a\in A_x$ such that $n(a)=p_1\cdots p_J$ satisfies the prime factor $p_j\in [x^{1/(s+1)}, x^{1/s})$ and $J-j\le (s+1)\beta$ with $s< \frac{\log x}{2\log\log x}$. Apparently for such $n(a)$, we have
    \begin{align}\label{eqn_d^r(n(p)) for S_3}
        d(n(a))^r \le d(p_1\cdots p_j)^r 2^{r(s+1)\beta}.
    \end{align}
    By symmetry, $d(n) \le 2\cdot \#\{d\mid n: d\ge \sqrt{n}\}$.  Applying $p_1\cdots p_j>x^{c/2}$ and $p_i< x^{1/s}$ for $1\le i\le j$, we infer that 
    \begin{align}\label{eqn_d^r(p_1...p_j) for S_3}
        d(p_1\cdots p_j)^r \ll_r \sum_{\substack{d_1,\ldots,d_r \mid p_1\cdots p_j \\ d_1,\ldots,d_r \geq x^{c/4}}}1=\sum_{\substack{d_1,\ldots,d_r \in S(x^c,x^{1/s}) \\d_1,\ldots,d_r \mid p_1\cdots p_j \\ d_1,\ldots,d_r \geq x^{c/4}}}1\leq \sum_{\substack{d_1,\ldots,d_r \in S(x^c,x^{1/s}) \\n(a) \equiv 0 \pmod{\lcm(d_1,\ldots,d_r)} \\ d_1,\ldots,d_r \geq x^{c/4}}}1.
    \end{align}
    With \eqref{eqn_d^r(n(p)) for S_3} and \eqref{eqn_d^r(p_1...p_j) for S_3}, it follows that 
    \begin{align*}
    S_3
        &\ll_r \sum_{s}2^{r(s+1)\beta}\sum_{a\in A_x}\sum_{\substack{d_1,\ldots,d_r \in S(x^c,x^{1/s}) \\n(a) \equiv 0 \pmod{\lcm(d_1,\ldots,d_r)} \\ d_1,\ldots,d_r \geq x^{c/4}}}1 \\
        &\ll_{r,\beta}   \sum_s 2^{rs\beta} \sum_{\substack{d_1,\ldots,d_r \in S(x^c,x^{1/s}) \\ d_1,\ldots,d_r \geq x^{c/4}}} \pi(x,\lcm(d_1,\ldots,d_r)) \\
        &\ll \# A_x \sum_s 2^{rs\beta}\sum_{\substack{d_1,\ldots,d_r \in S(x^c,x^{1/s}) \\ d_1,\ldots,d_r \geq x^{c/4}}}\frac{1}{\phi(\lcm(d_1,\ldots,d_r))}.
    \end{align*}

    In light of the condition $s\ge 4/c$ in Proposition~\ref{prop4.4}, we handle the inner sum with two treatments.  When  $s < c/4$, we relax the constraints in the inner multiple sum and majorize it by
    \begin{align*}
        \sum_{d_1,\ldots,d_r\le x^c}\frac{1}{\phi(\lcm(d_1,\ldots,d_r))}. 
    \end{align*}
    By Proposition~\ref{prop_1/phi(lcm) formula}, this sum is  $\asymp (\log x)^{2^r-1}$ which is sufficient for our purpose.  While for $s \geq c/4$, we apply  the Cauchy-Schwarz inequality to get
    \begin{align}\label{eqn_1/phi(lcm) Cauch-Schwartz}
        \sum_{\substack{d_1,\ldots,d_r \in S(x^c,x^{1/s}) \\ d_1,\ldots,d_r \geq x^{c/4}}}\frac{1}{\phi(\lcm(d_1,\ldots,d_r))} &\leq \left(\sum_{\substack{d_1,\ldots,d_r \in S(x^c,x^{1/s}) \\ d_1,\ldots,d_r \geq x^{c/4}}}\frac{1}{\lcm(d_1,\ldots,d_r)}\right)^{1/2} \notag \\
        &\times \left(\sum_{d_1,\ldots,d_r \leq x^c}\frac{\lcm(d_1,\ldots,d_r)}{\phi(\lcm(d_1,\ldots,d_r))^2}\right)^{1/2}.
    \end{align}
    As $4/c < s < \frac{\log x}{2\log\log x}$, we can apply Proposition~\ref{prop4.4}. Together with  Proposition~\ref{prop_lcm/phi(lcm)^2 formula},  the right-hand side of \eqref{eqn_1/phi(lcm) Cauch-Schwartz} is majorized by $$(\log x)^{2^r-1}\exp \left(-\frac{cs\log s}{32}\right).$$
    Hence we obtain  
    \begin{align*}
        S_3 &\ll \# A_x (\log x)^{2^r-1} \sum_s 2^{rs\beta}\exp \left(-\frac{cs\log s}{32}\right) \ll \# A_x (\log x)^{2^r-1},
    \end{align*}
    for $2^{rs\beta}= \exp(sO(1))$.

    It remains to estimate $S_4$. Now $n(a)= p_1\cdots p_J$ and $s \ge \frac{\log x}{2\log \log x}$.  By the well-known estimate
    $$d(n) \ll \exp\left(\frac{\log n}{\log\log n}\right)$$
    for $n\geq 3$ (and $n(a)< x^\beta)$, we have
    \begin{align*}
        d(n(a))^r \le d(p_{j+1}\cdots p_J)^r \sum_{\substack{d_1,\ldots,d_r \in S(x^c,x^{1/s}) \\d_1,\ldots,d_r \mid p_1\cdots p_j \\ d_1,\ldots,d_r \geq x^{c/4}}}1 \ll \exp\left(\frac{r\beta\log x}{\log\log x}\right)\sum_{\substack{d_1,\ldots,d_r \in S(x^c,x^{1/s}) \\n(a) \equiv 0 \pmod{\lcm(d_1,\ldots,d_r)} \\ d_1,\ldots,d_r \geq x^{c/4}}}1
    \end{align*}
    From $(2\le)$ $ p_j< x^{1/s}$, we have $s \leq (\log x)/\log 2$. We deduce that
    \begin{align*}
        S_4 \ll \#A_x \sum_{\frac{\log x}{2\log\log x} \leq s \leq \frac{\log x}{\log 2}}\exp\left(\frac{r\beta \log x}{\log\log x}\right)\sum_{\substack{d_1,\ldots,d_r \in S(x^c,x^{1/s}) \\ d_1,\ldots,d_r \geq x^{c/4}}}\frac{1}{\phi(\lcm(d_1,\ldots,d_r))}.
    \end{align*}
    Recalling the well-known inequality $m/\phi(m)\ll \log\log m$, we get
    \begin{align*}
        S_4 \ll \#A_x  \sum_{\frac{\log x}{2\log\log x} \leq s \leq \frac{\log x}{\log 2}}\exp\left(\frac{r\beta \log x}{\log\log x}\right)(\log\log x)\sum_{\substack{d_1,\ldots,d_r \in S(x^c,x^{1/s}) \\ d_1,\ldots,d_r \geq x^{c/4}}}\frac{1}{\lcm(d_1,\ldots,d_r)}.
    \end{align*}
    Using Proposition~\ref{prop4.4}, the right-hand side of the above is
    \begin{align*}
        &\ll \#A_x \cdot x^{-\delta} \sum_{\frac{\log x}{2\log\log x} \leq s \leq \frac{\log x}{\log 2}}\exp\left(\frac{r\beta \log x}{\log\log x}\right)(\log\log x) 
        \ll \#A_x.
    \end{align*}

Our proof of Theorem~\ref{thm_main theorem general} is now complete.

\section{Preliminaries on $L$-functions}\label{sec_preliminaries on L-functions}

\subsection{Automorphic $L$-functions}\label{sec_standard L-functions}

Let us briefly recall some definitions and analytic properties of automorphic $L$-functions. For more detailed descriptions, we refer the reader to \cite{ST19} and the references contained therein. 

Consider a cuspidal automorphic representation $\pi=\bigotimes_{p<\infty}\pi_p$ of $\mathrm{GL}_n(\mathbb A_{\mathbb Q})$ with unitary central character. We assume that $\pi$ is normalized so that it is trivial on the diagonal embedded copy of $\mathbb R^{+}$. We denote its contragredient by $\tilde \pi$ and its conductor by $q(\pi)$. The $L$-function of $\pi$ is given by
\begin{align*}
    L(s,\pi):=\prod_{p < \infty}L(s,\pi_p).
\end{align*}
This product converges absolutely for $\Re(s)>1$. For each $p$, there exist complex numbers $\alpha_{\pi,1}(p), \ldots, \alpha_{\pi,n}(p)$ (called the Satake parameters) such that the local $L$-function $L(s,\pi_p)$ is given by
\begin{align*}
    L(s,\pi_p)=\prod_{i=1}^n \left(1-\frac{\alpha_{\pi,i}(p)}{p^s}\right)^{-1}.
\end{align*}
Note that the Satake parameters $\alpha_{\pi,i}(p)$ are non-zero if $\pi_p$ is unramified.

At the archimedean place $\infty$, the local $L$-factor $L_{\infty}(s,\pi)=L(s,\pi_{\infty})$ (called the gamma factor of $L(s,\pi)$) is given by
\begin{align*}
    L_{\infty}(s,\pi)=\prod_{i=1}^n \Gamma_{\mathbb R}(s+\mu_{\pi,i}),
\end{align*}
where $\Gamma_{\mathbb R}(s):=\pi^{-s/2}\Gamma(s/2)$ and $\mu_{\pi,1},\ldots,\mu_{\pi,n}$ are the Langlands parameters. It is known that we have the bounds (\cite{LRS99})
\begin{align*}
    |\alpha_{\pi,i}(p)| \leq p^{\theta_n}, \quad \Re(\mu_{\pi,i}) \geq -\theta_n
\end{align*}
for some $0 \leq \theta_n \leq \frac{1}{2}-\frac{1}{n^2+1}$. The (generalized) Ramanujan-Petersson conjecture predicts that $\theta_n$ can be zero.

To further describe $\mu_{\pi, i}$, let us introduce more descriptions of the archimedean $L$-factors. Langlands proved that there exist $1$- or $2$-dimensional irreducible representations $\phi_a$ of the Weil group $W_{\mathbb R}$ such that
\begin{align*}
    \pi_{\infty}=\bigoplus_{a \in \mathcal A}\phi_a.
\end{align*}
It follows that
\begin{align*}
    L_{\infty}(s,\pi)=\prod_{a \in \mathcal A}L(s,\phi_a),
\end{align*}
and each of $L(s,\phi_a)$ can be described as follows.

\begin{enumerate}
    \item If $\phi$ is $1$-dimensional, then there exist $\nu \in \mathbb C$, $\epsilon \in \{0,1\}$ such that
    \begin{align}\label{eqn_L(s,phi) for 1-dim}
        L(s,\phi)=\Gamma_{\mathbb R}(s+\nu+\epsilon).
    \end{align}
    \item If $\phi$ is $2$-dimensional, then there exist $\nu \in \mathbb C$, $\kappa \in \mathbb Z$ such that
    \begin{align}\label{eqn_L(s,phi) for 2-dim}
        L(s,\phi)=\Gamma_{\mathbb C}(s+\nu+|\kappa|/2)=\Gamma_{\mathbb R}(s+\nu+|\kappa|/2)\Gamma_{\mathbb R}(s+\nu+|\kappa|/2+1).
    \end{align}
\end{enumerate}
In both cases, the bound $|\Re(\nu)|<1/2$ holds (\cite[Appendix A.3]{RS96}).

The complete $L$-function of $\pi$ is defined as
\begin{align*}
    \Lambda(s,\pi):=L_{\infty}(s,\pi)L(s,\pi).
\end{align*}
The function $\Lambda(s,\pi)$ extends to an entire function of order $1$ if $\pi$ is not trivial. If $\pi$ is trivial, then $L(s,\pi)$ is the Riemann zeta function, which has a simple pole at $s=1$. The function $\Lambda(s,\pi)$ satisfies the following functional equation:
\begin{align*}
    \Lambda(s,\pi)=\epsilon(\pi)q(\pi)^{1/2-s}\Lambda(1-s,\tilde \pi).
\end{align*}

For $t \in \mathbb R$, we define $\mathfrak q(t,\pi)$ as
\begin{align*}
    \mathfrak q(t,\pi):=q(\pi)\prod_{i=1}^n (1+|{\rm i}t+\mu_{\pi,i}|).
\end{align*}
The analytic conductor of $\pi$ is defined by
\begin{align*}
    \mathfrak q(\pi):=\mathfrak q(0,\pi).
\end{align*}

\subsection{Artin $L$-functions}\label{sec_Artin L-ftn}

Suppose $K/\mathbb Q$ is a finite Galois extension, and let $G$ denote its Galois group. For a prime $\mathfrak p$ above $p$, we let $\sigma_{\mathfrak p}$ be the Frobenius element at $\mathfrak p$, and $I_{\mathfrak p}$ be the inertia group for $\mathfrak p$. For a finite-dimensional complex representation $\rho: G \to \mathrm{GL}(V)$ of $G$ and its character $\chi$, the Artin $L$-function $L(s,\chi,K/\mathbb Q)$ is defined as follows,
\begin{align*}
    L(s,\chi,K/\mathbb Q)=\prod_p \det(1-\rho(\sigma_{\mathfrak p})p^{-s}|V^{I_{\mathfrak p}})^{-1}.
\end{align*}
In particular, for unramified $p$ in $K$, the local $L$-factor is $L_p(s,\chi,K/\mathbb Q)=\det(1-\rho(\sigma_p)p^{-s})^{-1}$. Note that since $\rho$ is $\chi(1)$-dimensional, $\det(1-\rho(\sigma_p)p^{-s})$ is a polynomial in $p^{-s}$ of degree $\chi(1)$.

The gamma factor for the Artin $L$-function is given by
\begin{align*}
    L_{\infty}(s,\chi,K/\mathbb Q)=\Gamma_{\mathbb R}(s)^{n^+(\chi)}\Gamma_{\mathbb R}(s+1)^{n^-(\chi)},
\end{align*}
where $n^+(\chi)$ and $n^-(\chi)$ are the dimensions of $(+)$- and $(-)$-eigenspace of $\rho(c)$, respectively, for $c$ the complex conjugation. The conductor of the Artin $L$-function is defined as
\begin{align*}
    q(\chi):=\mathrm{Norm}_{K/\mathbb Q}(n(\rho)),
\end{align*}
where $n(\rho)$ is the global Artin conductor. One has the inequality (\cite{MMS88}, \cite{Se81}. See also \cite[(2.6)]{Wo19}):
\begin{align*}
    \log q(\chi) \ll \chi(1)\mathcal M(K/\mathbb Q),
\end{align*}
where
\begin{align}\label{M}
    \mathcal M(K/\mathbb Q):=[K:\mathbb Q]\prod_{p \text{ ramified in } K}p.
\end{align}

The Langlands reciprocity conjecture predicts that for any irreducible character $\chi$, there is an attached cuspidal automorphic representation of $\mathrm{GL}_{\chi(1)}(\mathbb A_{\mathbb Q})$, say $\pi_{\chi}$, that satisfies
\begin{align*}
    L(s,\chi,K/\mathbb Q)=L(s,\pi_{\chi}).
\end{align*}
Note that in this situation, $q(\pi_{\chi})=q(\chi)$ and the Langlands parameter $\mu_{\pi_{\chi},i}$ is either $0$ or $1$.

\subsection{$L$-functions of $\mathrm{Sym}^m f$}\label{sec_L ftns on sym}

Let $f$ be a non-CM newform of weight $k$ and level $N$ with trivial character. 
The $L$-function associated with $f$ is given by
\begin{align*}
    L(s,f)=\sum_{n=1}^{\infty}\frac{\lambda_f(n)}{n^s}=\prod_{p}(1-\alpha_p p^{-s})^{-1}(1-\beta_p p^{-s})^{-1}.
\end{align*}
Associated with $f$, we have  a regular algebraic cuspidal automorphic representation $\pi_f$ of $\mathrm{GL}_2(\mathbb A_{\mathbb Q})$ and an ($\ell$-adic) Galois representations $\rho_f$ of $\mathrm{Gal}(\overline{\mathbb Q}/\mathbb Q)$. It is known that $L(s,f)= L(s, \pi_f)=L(s,\rho_f)$.
Its symmetric power $L$-functions are of the form
\begin{align*}
    L(s,\mathrm{Sym}^m f):=\prod_{p \mid N}L_p(s,\mathrm{Sym}^m f)\prod_{p \nmid N}\prod_{i=0}^m (1-\alpha_p^i \beta_p^{m-i}p^{-s})^{-1},
\end{align*}
where $\prod_{p \mid N}L_p(s,\mathrm{Sym}^m \pi_f)$ accounts for the local factors at ramified primes. Indeed, one may define $L(s,\mathrm{Sym}^m f):=L(s,\mathrm{Sym}^m \rho_f)$. Due to the breakthrough result on the Langlands functoriality of symmetric powers by Newton and Thorne \cite{NT21a, NT21b}, it was proven that there exists an automorphic representation $\mathrm{Sym}^m(\pi_f)$ of $\mathrm{GL}_{m+1}(\mathbb A_{\mathbb Q})$ which gives a Galois representation $\mathrm{Sym}^m(\rho_f)$ via the Langlands correspondence. Consequently, $L(s,\mathrm{Sym}^m f) = L(s,\mathrm{Sym}^m \pi_f) = L(s,\mathrm{Sym}^m \rho_f)$ have nice analytic properties as in Section~\ref{sec_standard L-functions}. In details, we have the following. 
\begin{enumerate}
    \item The conductor $q(\mathrm{Sym}^m f)$ satisfies
    \begin{align*}
        q(\mathrm{Sym}^m f)=O(N^{am}),
    \end{align*}
    where $a$ is a constant not depends on $m$ (\cite[$\mathsection$~5]{Ro07}).
    \item The gamma factor of $L(s,\mathrm{Sym}^m f)$ is given by
    \begin{align*}
        L_{\infty}(s,\mathrm{Sym}^m f)=\begin{cases}
            \prod_{i=1}^{\frac{m+1}{2}}\Gamma_{\mathbb C}\left(s+\left(i-\frac{1}{2}\right)(k-1)\right), & \text{if $m$ is odd}, \\
            \Gamma_{\mathbb R}(s+\epsilon)\prod_{i=1}^{\frac{n}{2}}\Gamma_{\mathbb C}(s+i(k-1)), & \text{if $m$ is even}.
        \end{cases}
    \end{align*}
    \item Let $\delta(m)=0$ if $m=0$ and $\delta(m)=1$ otherwise. The function $s^{\delta(m)}(s-1)^{\delta(m)}\Lambda(s,\mathrm{Sym}^m f)$ extends to an entire function on $\mathbb C$ and satisfies the functional equation
    \begin{align*}
        \Lambda(s,\mathrm{Sym}^m f)=\epsilon(\mathrm{Sym}^m f)q(\mathrm{Sym}^m f)^{1/2-s}\Lambda(1-s,\mathrm{Sym}^m f).
    \end{align*}
\end{enumerate}

\subsection{Rankin-Selberg $L$-functions}\label{sec_Rankin-Selberg L ftn}

Let $\pi$ and $\pi'$ be cuspidal automorphic representations of $\mathrm{GL}_n(\mathbb A_{\mathbb Q})$ and $\mathrm{GL}_{n'}(\mathbb A_{\mathbb Q})$, respectively, both with unitary central character and be normalized as in $\mathsection$~\ref{sec_standard L-functions}. For each prime $p$, the local Rankin-Selberg factor is
\begin{align*}
    L(s,\pi_p \times \pi'_p)=\prod_{i=1}^n \prod_{j=1}^{n'}\left(1-\frac{\alpha_{\pi \times \pi', i, j}(p)}{p^s}\right)^{-1}.
\end{align*}
If both of $\pi_p$ and $\pi_p'$ are unramified, then the Satake parameter $\alpha_{\pi \times \pi', i, j}(p)$ is given by
\begin{align*}
    \alpha_{\pi \times \pi', i, j}(p)=\alpha_{\pi,i}(p)\alpha_{\pi',j}(p).
\end{align*}
We refer the reader to \cite[Appendix]{ST19} for a complete description of $\alpha_{\pi \times \pi', i, j}(p)$. The Rankin-Selberg $L$-function is an $L$-function associated with $\pi$ and $\pi'$, which is given by
\begin{align*}
    L(s,\pi \times \pi')=\prod_p L(s,\pi_p \times \pi'_p).
\end{align*}
The function $L(s,\pi \times \pi')$ converges absolutely for $\Re(s)>1$, as well as $L(s,\pi)$. Denote its conductor and analytic conductor by $q(\pi \times \pi')$ and $\mathfrak q(\pi \times \pi')$, respectively. By Bushnell and Henniart \cite{BH97} and Brumley \cite[Appendix]{HB19}, one has
\begin{align*}
    q(\pi \times \pi') \mid q(\pi)^{n'}q(\pi')^n \quad \text{and} \quad \mathfrak q(\pi \times \pi') \ll \mathfrak q(\pi)^{n'}\mathfrak q(\pi')^n.
\end{align*}

Moving on to the archimedean places, the gamma factor is given by
\begin{align*}
    L_{\infty}(s,\pi \times \pi')=\prod_{i=1}^n\prod_{j=1}^{n'}\Gamma_{\mathbb R}(1+\mu_{\pi \times \pi',i,j}),
\end{align*}
which is described as follows: let $\pi_{\infty}=\bigoplus_{a \in \mathcal A}\phi_a$ and $\pi'_{\infty}=\bigoplus_{b \in \mathcal B}\phi_b'$, where $\phi_a$ and $\phi_b'$ are $1$- or $2$-dimensional irreducible representations of the Weil group $W_{\mathbb R}$. Let $\nu$, $\nu'$, $\epsilon$, $\epsilon'$, $\kappa$, and $\kappa'$ be given by \eqref{eqn_L(s,phi) for 1-dim}, \eqref{eqn_L(s,phi) for 2-dim} for $\phi$ and $\phi'$.
\begin{enumerate}
    \item If $\phi$ and $\phi'$ are $1$-dimensional, then
    \begin{align*}
        L(s,\phi \otimes \phi')=\Gamma_{\mathbb R}(s+\nu+\nu'+\epsilon''),
    \end{align*}
    where $\epsilon'' \in \{0,1\}$ satisfies $\epsilon'' \equiv \epsilon+\epsilon' \pmod 2$. 
    \item If $\phi$ is $1$-dimensional and $\phi'$ is $2$-dimensional, then
    \begin{align*}
        L(s,\phi \otimes \phi')=\Gamma_{\mathbb C}(s+\nu+\nu'+|\kappa'|/2).
    \end{align*}
    \item If $\phi$ and $\phi'$ are $2$-dimensional, then
    \begin{align*}
        L(s,\phi \otimes \phi')=\Gamma_{\mathbb C}(s+\nu+\nu'+|\kappa+\kappa'|/2)\Gamma_{\mathbb C}(s+\nu+\nu'+|\kappa-\kappa'|/2).
    \end{align*}
\end{enumerate}
The function $L_{\infty}(s,\pi \times \pi')$ is given by
\begin{align*}
    L_{\infty}(s,\pi \times \pi')=\prod_{a \in \mathcal A}\prod_{b \in \mathcal B}L(s,\phi_a \otimes \phi_b).
\end{align*}
From the descriptions of the Satake and Langlands parameters, one has
\begin{align*}
    \alpha_{\pi \times \pi', i, j} \leq p^{\theta_n+\theta_{n'}}, \quad \Re(\mu_{\pi \times \pi', i, j}) \geq -\theta_n-\theta_{n'}
\end{align*}
and in particular, the Ramanujan--Petersson conjecture for $L(s,\pi \times \pi')$ will follow if it holds for $L(s,\pi)$ and $L(s,\pi')$.

The complete $L$-function
\begin{align*}
    \Lambda(s,\pi \times \pi')=L_{\infty}(s,\pi \times \pi')L_{\infty}(s,\pi \times \pi')
\end{align*}
extends to a meromorphic function on $\mathbb C$. If $\pi \not\simeq \widetilde{\pi}'$, the function $\Lambda(s,\pi \times \pi')$ is entire, and if $\Lambda(s,\pi \times \pi')$ admits poles, then such poles must be simple and arise at $s=0,1$. The function $\Lambda(s,\pi \times \pi')$ satisfies the functional equation
\begin{align*}
    \Lambda(s,\pi \times \pi')=\epsilon(\pi \times \pi')q(\pi \times \pi')^{1/2-s}\Lambda(1-s,\tilde{\pi} \times \tilde{\pi}').
\end{align*}

\subsection{$L$-functions of $\mathrm{Sym}^m f \otimes \chi$ and Hypothesis A}\label{sec_L ftns on sym otimes chi}

Let $f$ be a non-CM newform of weight $k$ and level $N$ with trivial character. 
Suppose $\chi$ is a non-trivial irreducible character of $\mathrm{Gal}(K/\mathbb Q)$ as described in $\mathsection$~\ref{sec_Artin L-ftn}. (When $\chi$ is trivial, $L(s,\mathrm{Sym}^m f \otimes \chi)=L(s,\mathrm{Sym}^m f)$.) The Langlands reciprocity conjecture predicts that there exists a cuspidal automorphic representation $\pi_{\chi}$ of $\mathrm{GL}_{\chi(1)}(\mathbb A_{\mathbb Q})$ associated with $\chi$. With  $\mathsection$~\ref{sec_Rankin-Selberg L ftn}, we infer the following (hypothetical) properties of the Rankin-Selberg $L$-function $L(s,\mathrm{Sym}^m f \otimes \chi)$. 
\begin{hypothesis}\label{hypo_nice analytic properties of L-function}
    Let $K/\mathbb Q$ be a finite Galois extension with Galois group $G$. Suppose $G$ is isomorphic to a subgroup of $\mathrm{GL}_2(\mathbb Z/\delta Z)$ for some $\delta \in \mathbb N$. Let $\chi$ be an irreducible character of $\rho_{\chi}$, where $\rho_{\chi}$ is a finite-dimensional complex representation of $\mathrm{Gal}(K/\mathbb Q)$. Then the $L$-function $L(s,\mathrm{Sym}^m f \otimes \chi)=L(s,\mathrm{Sym}^m \rho_f \otimes \rho_{\chi})$ satisfies the following properties.
    \begin{enumerate}
        \item The equation for the gamma factor of $L(s,\mathrm{Sym}^m f \otimes \chi)$ is
        \begin{align*}
            & L_{\infty}(s,\mathrm{Sym}^m f \otimes \chi)\\
            &=\begin{cases}
                \prod_{i=1}^{(m+1)/2}\Gamma_{\mathbb C}\left(s+\left(i-\frac{1}{2}\right)(k-1)\right)^{\chi(1)}, & \text{if $m$ is odd,} \\ \Gamma_{\mathbb R}(s+\epsilon)^{n^+(\chi)}\Gamma_{\mathbb R}(s+\epsilon')^{n^-(\chi)}\prod_{i=1}^{m/2}\Gamma_{\mathbb C}(s+i(k-1))^{\chi(1)}, &\text{if $m$ is even,}
            \end{cases}
        \end{align*}
        where $\Gamma_{\mathbb R}(s):=\pi^{-s/2}\Gamma(s/2)$, $\Gamma_{\mathbb C}(s):=\Gamma_{\mathbb R}(s)\Gamma_{\mathbb R}(s+1)$, and $\epsilon,\epsilon' \in \{0,1\}$, $\epsilon+\epsilon' \equiv 1 \pmod 2$, $\epsilon \equiv m/2 \pmod 2$.
        \item Define $\Lambda(s,\mathrm{Sym}^m f \otimes \chi):=L_{\infty}(s,\mathrm{Sym}^m f \otimes \chi)L(s,\mathrm{Sym}^m f \otimes \chi)$ and let $\delta(m,\chi)=1$ if $m=0$ and $\chi$ is trivial, and $\delta(m,\chi)=0$ otherwise. Then the function $$s^{\delta(m,\chi)}(1-s)^{\delta(m,\chi)}\Lambda(s,\mathrm{Sym}^m f \otimes \chi)$$ extends to an entire function on $\mathbb C$ of order $1$ which does not vanish at $s=0,1$.
        \item The function $\Lambda(s,\mathrm{Sym}^m f \otimes \chi)$ satisfies the functional equation
        \begin{align*}
             & \Lambda(s,\mathrm{Sym}^m f \otimes \chi)=\epsilon(\mathrm{Sym}^m f \otimes \chi)q(\mathrm{Sym}^m f\otimes \chi)^{1/2-s}\Lambda(1-s,\mathrm{Sym}^m f \otimes \overline{\chi}),
        \end{align*}
        where $\epsilon(\mathrm{Sym}^m f \otimes \chi)$ is a complex number of absolute value $1$ (known as the root number) and $q(\mathrm{Sym}^m f \otimes \chi)$ is the (arithmetic) conductor of $\mathrm{Sym}^m f \otimes \chi$.
        \item The analytic conductor $\mathfrak q(\mathrm{Sym}^m f \otimes \chi)$ satisfies the inequality
        \begin{align*}
            \mathfrak q(\mathrm{Sym}^m f\otimes \chi) \leq \mathfrak q(\mathrm{Sym}^m f)^{\chi(1)}\mathfrak q(\chi)^{m+1}. 
        \end{align*}
    \end{enumerate}
\end{hypothesis}
\begin{remark}
   (1) {\it Hypothesis}~\ref{hypo_nice analytic properties of L-function}(3) follows from that the contragredient of $\mathrm{Sym}^m \pi_f \otimes \pi_{\chi}$ is given by $\mathrm{Sym}^m \pi_f \otimes \overline{\chi}$, because the character of a dual representation is given by the complex conjugation of the character, and $\mathrm{Sym}^m \pi_f$ is self-dual.

   (2) The Ramanujan-Petersson conjecture holds for $L(s,\mathrm{Sym}^m f \otimes \chi)$, as well as for $L(s,\mathrm{Sym}^m f)$ and $L(s,\chi)$.
\end{remark}

\section{Effective Chebotarev-Sato-Tate and   Theorem~\ref{prop_P(x,I;f;delta)}}\label{sec_effective CST}

Let $f$ be a non-CM newform of weight $k$ and level $N$ and recall the Sato-Tate conjecture \eqref{eqn_Sato Tate law}. Indeed we have
\begin{align*}
    \#\{p \leq x:\theta_f(p) \in I\}=\mu_{ST}(I)\pi(x)+o(\pi(x)).
\end{align*}
Under GRH, we have the following effective result for the Sato-Tate conjecture. 
\begin{theorem}{\cite[Theorem 1.3]{Th21}}
    If $L(s,\mathrm{Sym}^m f)$ satisfies GRH for all $m\geq 0$, then
    \begin{align*}
        |\#\{p \leq x:\theta_f(p) \in I\}-\mu_{ST}(I)\pi(x)| \ll x^{3/4}\frac{\log(kNx)}{\log x}.
    \end{align*}
\end{theorem}
We refer the reader to \cite[Theorem 1.1]{Th21} for an unconditional but weaker version of the effective Sato-Tate equidistribution. On the other hand, as mentioned in \cite[Lemma 5.2]{MM84}, Serre \cite[p. 175]{Se81} proved under GRH that
\begin{align}\label{eqn_a(p)=0 proportion}
    \#\{p \leq x:a_f(p)=0\} \ll x^{3/4}.
\end{align}
Combining these two bounds gives
\begin{align*}
    \#\prod(x,I;f)=\mu_{ST}(I)\pi(x)+O\left(x^{3/4}\frac{\log(kNx)}{\log x}\right).
\end{align*}
Using the explicit version of the Chebotarev density theorem \cite{LO77}, it was proved in \cite[Lemma 5.3]{MM84} that
\begin{align}\label{eqn_pi(x,delta,I=[0,pi])}
    \#\{p \leq x: a_f(p) \neq 0, a_f(p) \equiv 0 \pmod{\delta}\}=h_f(\delta)\pi(x)+O(\delta^3 x^{1/2}\log(\delta Nx))+O(x^{3/4}),
\end{align}
where $h_f(\delta)$ is the density of an associated conjugacy class, see below.

\subsection{Effective Chebotarev-Sato-Tate} The Chebotarev and Sato-Tate laws can be combined to address the density problem under the conditions that $a_f(p) \equiv 0 \pmod{\delta}$ and $\theta_f(p) \in I$ simultaneously. This  was first observed by Ram Murty and Kumar Murty \cite{MM09} for elliptic curves defined over a totally real field $K'$ with at least one prime of multiplicative reduction. If $K/K'$ is a finite solvable Galois extension and $C$ is a conjugacy class of $\mathrm{Gal}(K/K')$, then the density of prime ideals $\mathfrak p$ of $K'$ for which the Frobenius element $\sigma_{\mathfrak p} \in C$ and the angle $\theta_{\mathfrak p} \in I$ is
\begin{align*}
    \frac{|C|}{|G|}\mu_{ST}(I).
\end{align*}
It is conjectured that the above holds true for arbitrary finite extension $K/K'$, as predicted by the Langlands program (see \cite[$\mathsection~5$]{MM09}). For an effective version of the Chebotarev-Sato-Tate conjecture, it has been proven by Wong \cite{Wo19} under assumptions of the GRH and the analytic properties of $L$-functions.

\begin{theorem}{\cite[Theorem 1.2]{Wo19}}\label{thm_effective CST for Hilbert modular form}
    Let $K'$ be a totally real field and $K/K'$ be a finite Galois extension with Galois group $G$, and let $\pi$ be a non-CM Hilbert modular form of $K'$ with conductor $q=q(\pi)$. Suppose that all the irreducible characters $\chi$ of $G$ and all the symmetric powers $\mathrm{Sym}^m \pi$ for $m\geq 1$ of $\pi$ are cuspidal over $K'$, and that each $L(s,\mathrm{Sym}^m \pi \otimes \chi)$ satisfies GRH. Assume Serre's bound
    \begin{align*}
        q(\mathrm{Sym}^m \pi)=O(q^{am})
    \end{align*}
    for $\pi$, where $a$ is some constant independent of $\pi$. For a conjugacy class $C$ of $G$, we have
    \begin{align*}
        &\#\{\mathfrak p \text{ a prime of $K'$} : N\mathfrak p \leq x, \sigma_{\mathfrak p} \in C, \theta_{\mathfrak p}\in I\}-\frac{|C|}{|G|}\mu_{ST}(I)\pi(x) \\ 
        &\ll x^{3/4}|C|^{1/2}[K':\mathbb Q]^{3/2}\left(\log (x\mathcal M(K/K')q)\right)^{1/2},
    \end{align*}
    where $\mathcal M(K/K')$ is defined as 
    \begin{align*}
        M(K/K')=[K:K']D_{K'}^{1/[K':\mathbb Q]}\prod_{p \in P(K/K')}p.
    \end{align*}
    Here, $D_{K'}$ is the absolute discriminant of $K'$ and $P(K/K')$ is the set of rational primes $p$ for which there is a prime $v$ of $K'$ lying over $p$ such that $v$ is ramified in $K$.
\end{theorem}

\begin{remark}
    In \cite[Theorem 1.2]{Wo19}, the result is stated under the assumption that $K/K'$ is a Galois extension of totally real fields. However, it is implicitly assumed that $K$ is a finite Galois extension of $K'$ since the size of the Galois group $G$ is assumed to be finite. It should be noted that the proof of the theorem works not only for a totally real field $K$, but also for a general $K$. (In other words, it suffices only to assume that $K'$ is a totally real field.)
\end{remark}

Although Theorem~\ref{thm_effective CST for Hilbert modular form} assumes the automorphy of $\chi$ and $\mathrm{Sym}^m \pi$, the proof of the theorem relies on analytic properties of $L(s,\mathrm{Sym}^m \pi \otimes \chi)$, along with the Ramanujan-Petersson bound for $\mathrm{Sym}^m \pi$, Serre's bound, and GRH. If $K'=\mathbb Q$ and $\pi=f$ is a non-CM (elliptic) newform, then all of these properties are satisfied under the assumption of the analytic properties described in {\it Hypothesis}~\ref{hypo_nice analytic properties of L-function} and GRH. Therefore, we can state the following result (which is a specific case of Theorem~\ref{thm_effective CST for Hilbert modular form}) for clarity:

\begin{theorem}\label{thm_effective CST}
    Let $K/{\mathbb Q}$ be a finite Galois extension with Galois group $G$, $f$ be a non-CM newform of level $N$, $\chi$ be an irreducible character of $G$, and $C$ be a conjugacy class of $G$. Assuming that $L(s,\mathrm{Sym}^m f \otimes \chi)$ satisfies the analytic properties described in {\it Hypothesis}~\ref{hypo_nice analytic properties of L-function} and GRH, then we have
    \begin{align*}
        \#\{p \leq x : \sigma_p \in C, \theta_f(p) \in I\}-\frac{|C|}{|G|}\mu_{ST}(I)\pi(x) \ll x^{3/4}|C|^{1/2}\left(\log (x\mathcal M(K/\mathbb Q)N)\right)^{1/2}.
    \end{align*}
\end{theorem}

\subsection{Proof of Theorem~\ref{prop_P(x,I;f;delta)}}

For a non-CM newform $f$ whose $a_f(n)\in \mathbb{Z}$ and an interval $I \subseteq [0,\pi]$, we denote $\prod(x,I;f;\delta) :=  \#\{p\le x: 0\neq a_f(p)\equiv 0 \ ({\rm mod} \ \delta), \ \theta_f(p)\in I\}$.

For a positive integer $\delta$ we let 
\begin{align*}
    \rho_{f,\delta}\colon \mathrm{Gal}(\overline{\mathbb Q}/\mathbb Q) \to \mathrm{GL}_2\left(\prod_{\substack{\ell \mid \delta \\ \ell \text{  prime}}} \mathbb Z_{\ell}\right)=\prod_{\substack{\ell \mid \delta \\ \ell \text{  prime}}} \mathrm{GL}_2(\mathbb Z_{\ell}) 
\end{align*}
be the Galois representations attached to $f$, and denote by $\bar{\rho}_{f,\delta}$ its mod $\delta$-reduction, which is a map from $\mathrm{Gal}(\overline{\mathbb Q}/\mathbb Q)$ to $\mathrm{GL}_2(\mathbb{Z}/\delta\mathbb{Z})= \prod_{\substack{\ell^n \mid \delta}} \mathrm{GL}_2(\mathbb Z/ \ell^n\mathbb{Z})$. Let $K_{\delta}$ be a number field such that $\mathrm{Gal}(\overline{\mathbb Q}/K_{\delta})=\mathrm{ker}(\bar{\rho}_{f,\delta})$. We also let $G_{\delta}:=\mathrm{Gal}(K_{\delta}/\mathbb Q)\cong\mathrm{im}(\bar{\rho}_{f,\delta})$, $C_{\delta}$ be the subset of $G_{\delta}$ consisting of elements of trace $0$, and $h_f(\delta):=|C_{\delta}|/|G_{\delta}|$. Note that if $p \nmid \delta N$, then $a_f(p) \equiv 0 \pmod{\delta}$ is equivalent to  $\bar{\rho}_{f,\delta}(\sigma_p) \in C_{\delta}$, where $\sigma_p$ is any Frobenius element $\sigma_{\mathfrak p}$ with respect to a prime $\mathfrak p$ above $p$ (i.e., any representative of the conjugacy class of $\sigma_{\mathfrak p}$).

 Since $C_{\delta}$ is invariant under conjugation by $G_{\delta}$, we can write
    \begin{align*}
        C_{\delta}=C_1 \sqcup C_2 \sqcup \cdots\sqcup C_s
    \end{align*}
    for some conjugacy classes $C_i$ of $G_{\delta}$. Note that $a_f(p) \equiv 0 \pmod{\delta}$ if and only if $\sigma_p \in C_{\delta}$. For any interval $I \subseteq [0,\pi]$, we have
    \begin{align*}
        \pi^*(x,C_{\delta},I;f)=\pi^*(x,C_1,I;f) + \pi^*(x,C_2,I;f) + \cdots +\pi^*(x,C_s,I;f),  
    \end{align*}
    where $\pi^*(x,C_i,I;f):=\#\{p \leq x : \sigma_p \in C_i, \  \theta_f(p) \in I\}$. Theorem~\ref{thm_effective CST} implies that the right-hand side of the above is given by
    \begin{align*}
        h_f(\delta)\mu_{ST}(I)\pi(x)+O(x^{3/4}(|C_1|^{1/2}+|C_2|^{1/2}+\cdots+|C_s|^{1/2})\log^{1/2}(x\mathcal M(K/\mathbb Q)N)).
    \end{align*}
    Let $D_{K_{\delta}}$ be the discriminant of $K_{\delta}$. As $[K_{\delta}:\mathbb Q] \leq \delta^4$ and $\log D_{K_{\delta}} \leq [K_{\delta}:\mathbb Q]\log(\delta N[K_{\delta}:\mathbb Q])$ (cf. \cite[p. 73]{MM84}), we have
    \begin{align*}
        &|C_1|^{1/2}+|C_2|^{1/2}+\cdots+|C_s|^{1/2}  \leq |C_{\delta}| \leq h_f(\delta)\delta^4, \quad {\rm and}\\
        &\log \mathcal M(K_{\delta}/\mathbb Q) \leq \log[K_{\delta}:\mathbb Q]+\log |D_{K_{\delta}}| \ll \delta^4 \log (\delta^5 N),
    \end{align*}
    implying $\log (\mathcal M(K_{\delta}/\mathbb Q)N) \ll \delta^4 \log (\delta N)$. Therefore, we obtain
    \begin{align*}
        \pi^*(x,C_{\delta},I;f)=h_f(\delta)\mu_{ST}(I)\pi(x)+O(x^{3/4}h_f(\delta)\delta^6\log^{1/2}(\delta N))+O(x^{3/4}h_f(\delta)\delta^4 \log^{1/2}x)
    \end{align*}
    and hence
    \begin{align*}
        \#\prod (x,I;f;\delta)=h_f(\delta)\mu_{ST}(I)\pi(x)+O(x^{3/4}h_f(\delta)\delta^6\log^{1/2}(\delta N))+O(x^{3/4}h_f(\delta)\delta^4 \log^{1/2}x)
    \end{align*}
    by~\eqref{eqn_a(p)=0 proportion} and $h_f(\delta)\ge \delta^{-1}$, see \eqref{eqn_h formula} below. 

    According to Serre \cite{Se72}, Swinnerton-Dyer \cite{SD73}, Carayol \cite{Ca83} and Ribet \cite{Ri75, Ri85} (see Theorem~\ref{thm_Loeffler} for a more general result) and Hensel's lemma, as we may assume $\ell \nmid k-1$, the adelic Galois representation attached to $f$,
    \begin{align*}
        \rho_f \,: \mathrm{Gal}_{\mathbb Q} \to \mathrm{GL}_2\left(\prod_{\ell} \mathbb Z_{\ell}\right)=\prod_{\ell} \mathrm{GL}_2(\mathbb Z_{\ell}),
    \end{align*}
    have the open image contained in $\prod_{\ell} H_{\ell}$, where
    \begin{align*}
        H_{\ell}:=\{M \in \mathrm{GL}_2(\mathbb Z_{\ell}) :  \det M \ ({\rm mod} \ \ell) \in (\mathbb F_{\ell}^{\times})^{k-1}\}.    
    \end{align*} 
    In particular, if $\ell$ and $\ell'$ are distinct primes sufficiently large enough, we have
    \begin{align*}
        &\mathrm{im}(\rho_{f,\ell^m})=\mathrm{im}(\rho_{f,\ell})=H_{\ell}, \\
        &\mathrm{im}(\rho_{f,\ell^m \ell'^n})=H_{\ell} \times H_{\ell'}, 
    \end{align*}
    and hence, under the identification of $G_\delta$ with $\mathrm{im}(\overline{\rho}_{f,\delta})$ ($\subset \mathrm{GL}_2(\mathbb Z/\delta)$),
    \begin{align*}
        &G_{\ell^m}=A_{\ell^m}:=\{M \in \mathrm{GL}_2(\mathbb Z/\ell^m) : \det M \ ({\rm mod} \ \ell) \in (\mathbb F_{\ell}^{\times})^{k-1}\},\\
        &G_{\ell^m \ell'^n}=G_{\ell^m} \times G_{\ell'^n}.
    \end{align*}
    It is shown in \cite[$\mathsection$~3]{GM14} that for any odd prime,
    $$
    \frac{|\{\gamma \in A_{\ell^m}: \, {\rm tr}\, \gamma =0\}|}{|A_{\ell^m}|} = \ell^{-m+2}(\ell^2-1)^{-1}.
    $$ 
    Thus we have
    \begin{align}\label{eqn_h formula}
        h_f(\delta) \asymp_f \prod_{\substack{\ell^m \mid\mid \delta \\ \ell \gg_f 1}} \frac{1}{\ell^{m-2}(\ell^2-1)} 
    \end{align}
    (see Appendix~\ref{sec_appendix}) which completes the proof.
\begin{remark}
    In their paper \cite{GM14}, Gun and Murty used the fact that if $\ell$ and $\ell'$ are sufficiently large distinct primes, then $K_{\ell^m} \cap K_{\ell'^n}=\mathbb Q$, and from this they deduced the formula
    \begin{align}\label{eqn_h full formula asserted by Gun--Murty}
        h_f(\delta) = \prod_{\ell^m \mid\mid \delta} \frac{1}{\ell^{m-2}(\ell^2-1)}
    \end{align}
    when the prime factors of $\delta$ are $\gg_f 1$. 

    This fact takes advantage of level 1 eigenforms. When $f$ is of level $1$, the representation $\overline{\rho}_{f,\delta}$ is ramified only at primes dividing $\delta$, so the field $K_{\delta}$ associated with $\overline{\rho}_{f,\delta}$ is ramified only at those primes. This implies that, if $\delta$ and $\delta'$ are relatively prime, then $K_{\delta} \cap K_{\delta'}$ is unramified at all primes, and hence its discriminant is 1 and $K_{\delta} \cap K_{\delta'} = \mathbb{Q}$. Since $K_{\ell^m} \cap K_{\ell'^n}=\mathbb Q$ is equivalent to $G_{\ell^m \ell'^n}=G_{\ell^m} \times G_{\ell'^n}$, it follows that $h(\delta) = \prod_{\ell^m\| \delta} h(\ell^m)$. 
    
    For a higher level, however, it is no longer clear that the relation $K_{\ell^m} \cap K_{\ell'^n} = \mathbb{Q}$ holds for arbitrary distinct primes $\ell$ and $\ell'$, because $K_{\ell^m}$ is ramified at $\ell$ and probably at some other primes, resulting that  $K_{\ell^m} \cap K_{\ell'^n}$ may ramify at some primes. For this reason, we invoke the results of Loeffler to understand $G_\delta$ and refer the reader to Appendix~\ref{sec_appendix}.
\end{remark}

\section{Proof of Theorem~\ref{thm_main theorem for modular forms}}\label{sec_proof of main thm for modular forms}
Now we prove Theorem~\ref{thm_main theorem for modular forms} using Theorems~\ref{thm_main theorem general} and \ref{prop_P(x,I;f;delta)}. Let us set
$$
A :=\{ p\in \mathbb{P}:\, 0\neq a_f(p), \ \theta_f(p)\in I\},
$$
and $n(p):=a_f(p)$ (see \S~\ref{SMR} and \S~\ref{SAR}).  
Clearly (U1) is satisfied by the Ramanujan bound: $|a_f(p)|< p^\beta$ with $\beta = (k-1)/2+\varepsilon$. By Theorem~\ref{prop_P(x,I;f;delta)}, under GRH or Hypothesis A with GRH  according to $I=[0,\pi]$ or not, we obtain that for all  $x\ge x_0(I,f)$ and $1\leq \delta \leq x^{1/25}$, 
$$
\pi(x,\delta) = h_f(\delta) (1 + o(1))\mu_{ST}(I)\pi(x) \sim h_f(\delta) (1 + o(1)) \cdot \#A_x
$$
and the $o$-term satisfies $|o(1)|\le \frac12$, where  $x_0(I,f)$ is some large constant depending on $I$ and $f$. 
Observing 
$$
\frac1\delta \le \prod_{\ell^m\| \delta} \ell^{-m} (1-\ell^{-2})^{-1} \le
h_f(\delta) \ll_f \frac1{\phi(\delta)},
$$
the conditions (U2) and (L) are fulfilled for all $1\le \delta \le  x^{1/25}$. Our result follows readily. 

\subsection*{Acknowledgements} We would like to express our sincere gratitude to the referees for their careful reading and very helpful comments, and  our heartfelt thanks to Professor Chun-Yin Hui for his interest in and comments on this paper. YKL is supported by GRF (No. 17317822) and NSFC (No. 12271458).

\appendix
\section{Adelic representation attached to modular forms}\label{sec_appendix}

In this appendix we give a detailed proof of \eqref{eqn_h formula}. In fact, we establish a more general result (Theorem~\ref{thm_appendix main}), from which \eqref{eqn_h formula} follows by taking $r=1$ and $F[T_1]=T_1$.

Let $f_1,\ldots,f_r$ be non-CM cuspidal eigenforms of weight $k_1,\ldots,k_r\geq 2$ and level $N_1,\ldots,N_r$ with trivial character, and suppose that every Fourier coefficients of $f_i$'s are integers. We can attach the $\ell$-adic Galois representation associated to $f_i$
\begin{align*}
    \rho_{f_i,\ell}: \mathrm{Gal}_{\mathbb Q} \to \mathrm{GL}_2(\mathbb Q_{\ell})
\end{align*}
for each prime $\ell$. The product of $\rho_{f_i,\ell}$ defines an adelic Galois representation 
\begin{align*}
    \rho_{f_i}:\mathrm{Gal}_{\mathbb Q} \to \mathrm{GL}_2(\hat{\mathbb Q}),
\end{align*}
where $\hat{\mathbb Q}$ is the ring of finite adeles. We will prove the following theorem.


Up to conjugation (equivalence of representations), we may assume that the image of $\rho_{f_i,\ell}$ is contained in $\mathrm{GL}_2(\mathbb Z_{\ell})$ for each $\ell$ so that $\mathrm{im}(\rho_{f_i}) \subseteq \mathrm{GL}_2(\hat{\mathbb Z})$, where $\hat{\mathbb Z}:=\prod_{\ell} \mathbb Z_{\ell}$. Moreover, by the results of Momose \cite{Mo81} and Ribet \cite{Ri85} on the image of Galois representation attached to a modular form, we have the following theorem (see \cite[Theorem~2.2.2]{Lo17} for a more general result.)

\begin{theorem}
    Let $B_{f_i}=M_{2 \times 2}(\mathbb Q)$ a central simple algebra of degree $2$ over $\mathbb Q$. Then we have
    \begin{align*}
        \mathrm{im}(\rho_{f_i}) \subset B_{f_i}(\hat{\mathbb Q})^{\times} = \mathrm{GL}_2(\hat{\mathbb Q})
    \end{align*}
    and
    \begin{align*}
        \mathrm{im}(\rho_{f_i,p})=\{x\in \mathrm{GL}_2(\mathbb Z_p): \det x \in \mathbb (\mathbb Z_p^{\times})^{k_i-1}\}.
    \end{align*}
    Moreover, $\mathrm{im}(\rho_{f_i})$ is an open subgroup of 
    \begin{align*}
        G_{f_i}(\hat{\mathbb Q}):=\{x \in (B_{f_i} \otimes \hat{\mathbb Q})^{\times}: \mathrm{norm}_{B_{f_i}/\mathbb Q}(x)=\lambda^{k_i-1} \text{ for some $\lambda \in \hat{\mathbb Q}$}\},
    \end{align*}
    where $\mathrm{norm}_{B_{f_i}/\mathbb Q}$ denotes the reduced norm map, which coincides with the determinant map.
\end{theorem}

Let
\begin{align*}
    \rho:=\rho_{f_1}\times \cdots \times \rho_{f_r}: \mathrm{Gal}_{\mathbb Q} \longrightarrow \mathrm{GL}_2^r(\hat{\mathbb Q}):= \prod_{1\le j\le r} \mathrm{GL}_2(\hat{\mathbb Q})
\end{align*}
be a product of $\rho_{f_i}$'s and $\rho_{\ell}$ be its $\ell$-adic component. Let $G$ be the algebraic group over $\mathbb Q$ defined by
\begin{align*}
    G=\{(g_1,\ldots,g_r)\in B_{f_1}^{\times} \times \cdots \times B_{f_r}^{\times}: \mathrm{norm}_{B_{f_i}/\mathbb Q}(g_i)=\lambda^{k_i-1} \text{ for some $\lambda \in \mathbb G_m$}\}.
\end{align*}
We also let $\mathcal H_p:=\{(g_1,\ldots,g_r)\in\mathrm{GL}_2(\mathbb Z_p)^r: \mathrm{det}(g_i)=\lambda^{k_i-1} \text{ for some $\lambda \in \mathbb Z_p^{\times}$}\}$. The following theorem states the adelic open image result for $\rho$.

\begin{theorem}\label{thm_Loeffler}
    Suppose that $f_1,\ldots,f_r$ are non-CM and for $i\neq j$, $f_i \neq f_j \otimes \chi$ for any Dirichlet character $\chi$ (including the trivial character). Then $\mathrm{im}(\rho)$ is an open subgroup of $G(\hat{\mathbb Q})$. In particular, there exists a prime $q$ such that $\mathrm{im}(\rho_p)=\mathcal H_p$ if $p >q$.

\end{theorem}
\begin{proof}
    This follows from a special case of \cite[Theorem~3.4.2]{Lo17}.
\end{proof}

Thus the image of $\rho$ is an open subgroup $\mathcal B$ of
\begin{align*}
    \prod_{\ell \leq q}(G(\mathbb Q_{\ell}) \cap \mathrm{GL}_2^r(\mathbb Z_{\ell})) \times \prod_{p>q}\mathcal H_p = G(\hat{\mathbb Q}) \cap \mathrm{GL}_2^r(\hat{\mathbb Z}).
\end{align*}
The openness of $\mathcal B$ implies that it is  of the form
\begin{align*}
    \mathcal B=\mathcal U_0\times \prod_{p\gg 1}\mathcal H_p
\end{align*}
for some open subset $\mathcal U_0 \subseteq \prod_{\ell \ll 1}(G(\mathbb Q_{\ell}) \cap \mathrm{GL}_2^r(\mathbb Z_{\ell}))$.
Enlarging $q$ if necessary, we may assume 
\begin{align*}
\mathcal B=\mathcal U_0\times \prod_{p >q}\mathcal H_p.
\end{align*}
Moreover, for our later purpose, we may assume $q$ by necessary enlarging such that $q\geq \max_{1\leq i \leq r}\{k_i-1\} $.
Therefore, we have
\begin{align*}
    \mathrm{im}(\rho) = \mathrm{im}\bigg(\prod_{\ell \leq q}\rho_{\ell}\bigg) \times \mathrm{im}\bigg(\prod_{p>q}\rho_p\bigg)
\end{align*}
where $\mathrm{im}\left(\prod_{\ell \leq q}\rho_{\ell}\right)$ is an open subgroup of $\prod_{\ell \leq q}(G(\mathbb Q_{\ell}) \cap \mathrm{GL}_2^r(\mathbb Z_{\ell}))$. Indeed, an arbitrary element $((\rho_{\ell}(\sigma))_{\ell \leq q},(\rho_p(\sigma'))_{p>q})$ of $\mathrm{im}\left(\prod_{\ell \leq q}\rho_{\ell}\right) \times \mathrm{im}\left(\prod_{p>q}\rho_p\right)$  can be written as
\begin{align*}
    ((\rho_{\ell}(\sigma))_{\ell \leq q},(\rho_p(\sigma))_{p>q}) \cdot ((1)_{\ell \leq q},(\rho_p(\sigma^{-1}\sigma'))_{p>q}),
\end{align*}
where $((\rho_{\ell}(\sigma))_{\ell \leq q},(\rho_p(\sigma))_{p>q})$ and $((1)_{\ell \leq q},(\rho_p(\sigma^{-1}\sigma'))_{p>q})$ belong to $\mathrm{im}(\rho)$ and $\mathcal B$, respectively, so $((\rho_{\ell}(\sigma))_{\ell \leq q},(\rho_p(\sigma'))_{p>q})$ belongs to $\mathrm{im}(\rho)$. 

Note that $\prod_{p>q}\mathcal H_p = \mathrm{im}\left(\prod_{p>q}\rho_p\right)$, and for $\infty \geq q_2\geq q_1>q$,
\begin{align*}
    \mathrm{im}\bigg(\prod_{\ell \leq q}\rho_{\ell} \times \prod_{q_1\leq p <q_2}\rho_p\bigg)=\mathrm{im}\bigg(\prod_{\ell \leq q}\rho_{\ell}\bigg)\times \prod_{q_1 \leq p <q_2}\mathcal H_p.
\end{align*}

Let $\mathcal U:=\mathrm{im}\left(\prod_{\ell \leq q}\rho_{\ell}\right)$. It can be written as $\mathcal U' \cap \prod_{\ell \leq q} G(\mathbb Q_{\ell})$ for an open subgroup $\mathcal U'$ of $\prod_{\ell \leq q} \mathrm{GL}_2^r(\mathbb Z_{\ell})$. We have a natural injection
\begin{align*}
    \prod_{\ell \leq q}(G(\mathbb Q_{\ell}) \cap \mathrm{GL}_2^r(\mathbb Z_{\ell})) / \mathcal U \longrightarrow{} \prod_{\ell \leq q} \mathrm{GL}_2^r(\mathbb Z_{\ell})/\mathcal U',
\end{align*}
so
\begin{align*}
    \left[\prod_{\ell \leq q}(G(\mathbb Q_{\ell}) \cap \mathrm{GL}_2^r(\mathbb Z_{\ell})): \mathcal U\right] \leq \left[\prod_{\ell \leq q} \mathrm{GL}_2^r(\mathbb Z_{\ell}):\mathcal U'\right].
\end{align*}
Since $\prod_{\ell \leq q} \mathrm{GL}_2^r(\mathbb Z_{\ell})$ is a locally profinite group with a fundamental system of neighborhoods (of $1$) consisting of open compact subgroups with finite indices, we obtain
\begin{align*}
     \left[\prod_{\ell \leq q}(G(\mathbb Q_{\ell}) \cap \mathrm{GL}_2^r(\mathbb Z_{\ell})) \times \prod_{p>q} \mathcal H_p: \mathrm{im}(\rho)\right]=\left[\prod_{\ell \leq q}(G(\mathbb Q_{\ell}) \cap \mathrm{GL}_2^r(\mathbb Z_{\ell})): \mathcal U\right]<\infty.
\end{align*}
Denote this number by $M_{f_1,\ldots,f_r}$.

For a positive integer $\delta$, the mod $\delta$ Galois representation
\begin{align*}
    \bar{\rho}_{\delta}:\mathrm{Gal}_{\mathbb Q} \longrightarrow \mathrm{GL}_2^r(\mathbb Z/\delta\mathbb Z)=\prod_{\ell^n \mid\mid \delta} \mathrm{GL}_2^r(\mathbb Z/\ell^n \mathbb Z)
\end{align*}
is given by the mod $\delta$-reduction of $\rho$. Thus we may write
\begin{align*}
    \bar{\rho}_{\delta}=\prod_{\ell^n \mid\mid \delta} \bar{\rho}_{\ell^n}.
\end{align*}
Let $\delta'$ be the squarefree part of $\delta$, that is, the squarefree integer whose prime factors coincide with the prime factors of $\delta$. 
Write $\rho_{\delta'} := \prod_{\ell|\delta'} \rho_\ell$ and denote by $\mathcal A_{\ell^n}'$ the image of $G(\mathbb Q_{\ell}) \cap \mathrm{GL}_2^r(\mathbb Z_{\ell})$ under the mod $\ell^n$-reduction. Then, we have the following commutative diagram.

\[\begin{tikzcd}
	{\mathrm{im}(\rho_{\delta'})} & {} & \begin{array}{c} \prod_{\substack{\ell^n \mid\mid \delta \\ \ell \leq q}}(G(\mathbb Q_{\ell}) \cap \mathrm{GL}_2^r(\mathbb Z_{\ell})) \times \prod_{\substack{p^m \mid\mid \delta \\ p>q}} \mathcal H_p \end{array} \\
	\\
	{\mathrm{im}(\bar{\rho}_{\delta})} && \begin{array}{c} \prod_{\substack{\ell^n \mid\mid \delta \\ \ell \leq q}}\mathcal A_{\ell^n}' \times \prod_{\substack{p^m \mid\mid \delta \\ p>q}} \mathcal A_{p^m} \end{array}
	\arrow[hook, from=1-1, to=1-3]
	\arrow[two heads, from=1-1, to=3-1]
	\arrow[two heads, from=1-3, to=3-3]
	\arrow[hook, from=3-1, to=3-3]
\end{tikzcd}\]
Here, the vertical maps are mod $\delta$-reductions and
\begin{align*}
    \mathcal A_{p^m}:=\{(g_1,\ldots,g_r) \in \mathrm{GL}_2^r(\mathbb Z/p^m\mathbb Z):\det g_i=v^{k_i-1} \text{ for some $v \in (\mathbb Z/p^m\mathbb Z)^{\times}$}\}.
\end{align*}
As we have $(p,k_i-1)=1$ for $p>q$ under our enlargement of $q$, the mod $p^m$-reduction from $\mathcal H_p$ surjects to $\mathcal A_{p^m}$ according to Hensel's lemma.

From the diagram, we obtain a surjection
\begin{align*}
    \left(\prod_{\substack{\ell^n \mid\mid \delta \\ \ell \leq q}}(G(\mathbb Q_{\ell}) \cap \mathrm{GL}_2^r(\mathbb Z_{\ell})) \times \prod_{\substack{p^m \mid\mid \delta \\ p>q}} \mathcal H_p\right)/\mathrm{im}(\rho_{\delta'}) \longrightarrow \left(\prod_{\substack{\ell^n \mid\mid \delta \\ \ell \leq q}}\mathcal A_{\ell^n}' \times \prod_{\substack{p^m \mid\mid \delta \\ p>q}} \mathcal A_{p^m}\right)/\mathrm{im}(\bar{\rho_{\delta}})
\end{align*}
and consequently, 
\begin{align*}
    \left[\prod_{\substack{\ell^n \mid\mid \delta \\ \ell \leq q}}\mathcal A_{\ell^n}' \times \prod_{\substack{p^m \mid\mid \delta \\ p>q}} \mathcal A_{p^m}:\mathrm{im}(\bar{\rho_{\delta}})\right] &\leq \left[\prod_{\substack{\ell^n \mid\mid \delta \\ \ell \leq q}}(G(\mathbb Q_{\ell}) \cap \mathrm{GL}_2^r(\mathbb Z_{\ell})) \times \prod_{\substack{p^m \mid\mid \delta \\ p>q}} \mathcal H_p:\mathrm{im}(\rho_{\delta'})\right] \\
    &\leq \left[\prod_{\ell\leq q}(G(\mathbb Q_{\ell}) \cap \mathrm{GL}_2^r(\mathbb Z_{\ell})) \times \prod_{\substack{p^m \mid\mid \delta \\ p>q}} \mathcal H_p:\mathrm{im}(\rho_{\delta'})\times \prod_{\substack{\hat{\ell} \nmid \delta \\ \hat{\ell}\leq q}}\mathrm{im}(\rho_{\hat{\ell}})\right] \\
    &\leq \left[\prod_{\ell\leq q}(G(\mathbb Q_{\ell}) \cap \mathrm{GL}_2^r(\mathbb Z_{\ell})) \times \prod_{\substack{p^m \mid\mid \delta \\ p>q}} \mathcal H_p:\mathrm{im}\left(\prod_{\ell \leq q}\rho_{\ell} \times \prod_{\substack{p^m \mid\mid \delta \\ p> q}}\rho_{\hat{\ell}}\right)\right] \\
    &=\left[\prod_{\ell\leq q}(G(\mathbb Q_{\ell}) \cap \mathrm{GL}_2^r(\mathbb Z_{\ell})) \times \prod_{\substack{p^m \mid\mid \delta \\ p>q}} \mathcal H_p:\mathrm{im}\left(\prod_{\ell \leq q}\rho_{\ell}\right)\right]=M_{f_1,\ldots,f_r}.
\end{align*}
Altogether, we have
\begin{align}\label{eqn_leq for 1/|im|}
    \frac{1}{|\mathrm{im}(\bar{\rho}_{\delta})|} \leq M_{f_1,\ldots,f_r} \prod_{\substack{\ell^n \mid\mid \delta \\ \ell \leq q}}\frac{1}{|\mathcal A_{\ell^n}'|}\times \prod_{\substack{p^m \mid\mid \delta \\ p>q}}\frac{1}{|\mathcal A_{p^m}|}.
\end{align}

\begin{theorem}\label{thm_appendix main}
Let $F \in \mathbb Z[T_1,\ldots,T_r]$ be a linear polynomial. Define 
\begin{align*}
    \mathcal C_{\delta}:=\{(g_1,\ldots,g_r)\in\mathrm{im}(\bar{\rho_{\delta}}): F(\mathrm{tr}(g_1),\ldots,\mathrm{tr}(g_r))=0\}.
\end{align*}
and 
\begin{align*}
    h(\delta):=\frac{|\mathcal C_{\delta}|}{|\mathrm{im}(\bar{\rho}_{\delta})|}. 
\end{align*}    
Then, 
\begin{align*}
    h(\delta) \asymp_{f_1,\ldots,f_r} \prod_{\substack{p^m \mid\mid \delta \\ p>q}} h(p^m).
\end{align*}
\end{theorem}

One can easily verify that the natural map
\begin{align*}
    \mathcal C_{\delta_1 \delta_2} \longrightarrow \mathcal C_{\delta_1}\times \mathcal C_{\delta_2}
\end{align*}
is injective, so $|\mathcal C_{\delta_1 \delta_2}| \leq |\mathcal C_{\delta_1}||\mathcal C_{\delta_2}|$. By combining with \eqref{eqn_leq for 1/|im|}, we have
\begin{align*}
    h(\delta)\leq M_{f_1,\ldots,f_r} \prod_{\substack{\ell^n \mid\mid \delta \\ \ell \leq q}}\frac{|\mathcal C_{\ell^n}|}{|\mathcal A_{\ell^n}'|}\times \prod_{\substack{p^m \mid\mid \delta \\ p>q}}\frac{|\mathcal C_{p^m}|}{|\mathcal A_{p^m}|} \leq M_{f_1,\ldots,f_r} \prod_{\substack{p^m \mid\mid \delta \\ p>q}}\frac{|\mathcal C_{p^m}|}{|\mathcal A_{p^m}|}=M_{f_1,\ldots,f_r}\prod_{\substack{p^m \mid\mid \delta \\ p>q}}h(p^m).
\end{align*}
On the other hand, one has
\begin{align*}
    h(\delta)=|\mathcal C_{\delta}|/|\mathrm{im}(\bar{\rho}_{\delta})| \geq\mathcal |\mathcal C_{\delta}|\prod_{p^m \mid\mid \delta} \frac{1}{|\mathrm{im}(\bar{\rho}_{p^m})|} \geq \prod_{\substack{\ell^n \mid\mid \delta \\ \ell \leq q}} \frac{1}{|\mathrm{im}(\bar{\rho}_{\ell^n})|}\prod_{\substack{p^m \mid\mid \delta \\ p>q}}h(p^m)=:N_{f_1,\ldots,f_r}\prod_{\substack{p^m \mid\mid \delta \\ p>q}} h(p^m),
\end{align*}
where $N_{f_1,\ldots,f_r}$ is a positive constant depending on $f_1,\ldots, f_r$ only.

Therefore, 
\begin{align*}
    h(\delta) \asymp_{f_1,\ldots,f_r} \prod_{\substack{p^m \mid\mid \delta \\ p>q}} h(p^m),
\end{align*}
implying the formula in \eqref{eqn_h formula}. Our proof is thus complete.

\end{document}